\documentclass[11pt,reqno]{amsart}
\usepackage{amssymb}
\usepackage[all]{xy}
\newtheorem{thm}{Theorem}[section]
\newtheorem{lem}[thm]{Lemma}
\newtheorem{prop}[thm]{Proposition}
\newtheorem{cor}[thm]{Corollary}

\theoremstyle{definition}
\newtheorem{dfn}[thm]{Definition}
\newtheorem{ex}[thm]{Example}

\theoremstyle{remark}
\newtheorem{remark}[thm]{Remark}

\newcommand{\CA}{{\mathcal{A}}}

\newcommand{\CE}{{\mathcal{E}}}
\newcommand{\bE}{{\overline{\mathcal{E}}}}

\newcommand{\bF}{{\overline{\mathcal{F}}}}
\newcommand{\CEa}{\CE^{0,-}}

\newcommand{\bH}{{\overline{H}}}

\newcommand{\CL}{{\mathcal{L}}}
\newcommand{\CB}{{\mathcal{B}}}
\newcommand{\CO}{{\mathcal{O}}}

\newcommand{\QCB}{[\mathcal{B}]}

\newcommand{\af}{\alpha}
\newcommand{\bt}{\beta}
\newcommand{\gm}{\gamma}
\newcommand{\dt}{\delta}

\newcommand{\ld}{\lambda}

\newcommand{\sm}{\sigma}

\newcommand{\Om}{\Omega}

\begin{document}


\title[The gauge-invariant ideals of labelled graph $C^*$-algebras]
{The structure of gauge-invariant ideals of labelled graph $C^*$-algebras}

\author[J. A. Jeong]{Ja A Jeong$^{\dagger}$}
\thanks{Research partially supported by NRF-2009-0068619$^{\dagger}$
and Hanshin University$^{\ddagger}$.}
\address{
Department of Mathematical Sciences and Research Institute of Mathematics\\
Seoul National University\\
Seoul, 151--747\\
Korea} \email{jajeong\-@\-snu.\-ac.\-kr }

\author[S. H. Kim]{Sun Ho Kim$^{\dagger}$}
\address{
Department of Mathematical Sciences\\
Seoul National University\\
Seoul, 151--747\\
Korea} \email{hoya4200\-@\-snu.\-ac.\-kr }

\author[G. H. Park]{Gi Hyun Park$^{\ddagger}$}
\address{
Department of Mathematics\\
Hanshin University\\
Osan, 151--747\\
Korea} \email{ghpark\-@\-hanshin.\-ac.\-kr }




\begin{abstract}
In this paper, we consider the gauge-invariant ideal 
structure of a $C^*$-algebra $C^*(E,\CL,\CB)$ associated to a set-finite, 
receiver set-finite and weakly left-resolving labelled space $(E,\CL,\CB)$, 
where $\CL$ is a labelling map assigning an alphabet 
to each edge of the directed graph $E$ with no sinks. 
Under the assumption that an accommodating set $\CB$ is closed 
under taking relative complement, it is obtained 
that there is a one to one correspondence between the set of 
all hereditary saturated subsets of $\CB$ and the gauge-invariant 
ideals of $C^*(E,\CL,\CB)$. 
For this, we introduce a quotient labelled space $(E,\CL, \QCB_\textsc{R})$ 
arising from an equivalence relation $\sim_\textsc{R}$ on $\CB$ 
and show the existence of the $C^*$-algebra $C^*(E,\CL, \QCB_\textsc{R})$ generated by a universal 
representation of $(E,\CL, \QCB_\textsc{R})$. 
Also the gauge-invariant uniqueness theorem for $C^*(E,\CL, \QCB_\textsc{R})$
is obtained. 

For simple labelled graph $C^*$-algebras $C^*(E,\CL,\bE)$, where  
$\bE$ is the smallest accommodating set containing all the generalized vertices, 
it is observed that if for each vertex $v$ of $E$, 
a generalized vertex $[v]_l$ is finite for some $l$, 
then $C^*(E,\CL,\bE)$ is simple if and only if 
$(E,\CL,\bE)$ is strongly cofinal and disagreeable. 
This is done by examining the merged labelled graph $(F,\CL_F)$ of $(E,\CL)$ and 
the common properties that  $C^*(E,\CL,\bE)$ and $C^*(F,\CL,\bF)$ share.
\end{abstract}

\maketitle

\setcounter{equation}{0}

\section{Introduction}

In \cite{BP1}, Bates and Pask introduced
a class of $C^*$-algebras  associated to labelled graphs.
Their motivation was to simultaneously generalize
ultra graph $C^*$-algebras (\cite{To1, To2}) and
the shift space $C^*$-algebras (\cite{CM,Ma}).
A labelled graph $C^*$-algebra $C^*(E,\CL,\CB)$ is
the universal $C^*$-algebra generated
by a  family of partial isometries $s_a$
indexed by labels $a$
and projections $p_A$ indexed by vertex subsets $A$ in an
accommodating set $\CB$ satisfying certain conditions.
By definition $C^*(E,\CL,\CB)$ depends on the choice of an accommodating set $\CB$
as well as a labelled graph $(E, \CL)$, where
$\CL$ is a labelling map assigning a label to each edge of $E$.
 An accommodating set $\CB$ is a collection of vertex subsets ($\CB\subset 2^{E^0}$)
 containing the ranges of all labelled paths
 which is closed under finite unions, finite intersections,
 and relative ranges.
 Among accommodating sets  of a labelled graph $(E,\CL)$,
 the smallest one $\CEa$ was mainly dealt with in \cite{BP2}
 under the assumptions that $(E,\CL)$ is essential
 ($E$ has no sinks and no sources),
 set-finite and receiver set-finite (every $A\in \CEa$
 emits and receives only finitely many labelled edges).
 Some conditions on $(E,\CL,\CEa)$
 were investigated to explore the simplicity of $C^*(E,\CL,\CEa)$ in \cite{BP2}.
 Since the accommodating set $\CEa$ is not closed under taking relative complement
 in general,  it may not contain generalized vertices $[v]_l$
 despite these generalized vertices were used effectively in \cite{BP2} as
 the canonical spanning set of labelled graph $C^*$-algebras $C^*(E,\CL,\CEa)$.
  This led us to consider an alternative of $\CEa$ in \cite{JK},
  that is,  the smallest accommodating set  $\bE$ which is
  closed under taking relative complement
  (or equivalently, the smallest accommodating set containing
  all generalized vertices).
  It was then proven  that
  if $C^*(E,\CL,\bE)$ is simple, $(E,\CL,\bE)$ is strongly cofinal
  (\cite[Theorem 3.8]{JK}) and
  if in addition $\{v\}\in \bE$ for
  every vertex $v\in E^0$, the labelled space $(E,\CL,\bE)$ is disagreeable
  (\cite[Theorem 3.14]{JK}).
  Furthermore,  a slight modification of the proof of Theorem 6.4 in \cite{BP2},
  shows that if $(E,\CL,\bE)$ is strongly cofinal and disagreeable,
  the $C^*$-algebra $C^*(E,\CL,\bE)$ is simple (\cite[Theorem 3.16]{JK}).
  Even when $\CEa\neq \bE$, if $(E,\CL,\CEa)$ and $(E,\CL,\bE)$ are
  weakly left-resolving, $C^*(E,\CL,\CEa)\cong C^*(E,\CL,\bE)$
  (Corollary~\ref{cor-isomorphism}).

  By universal property, $C^*(E,\CL,\CB)$ admits
  the gauge action of the unit circle.
As for the gauge-invariant ideal structure of graph $C^*$-algebras,
it is known (\cite{BPRS}) that the set of gauge-invariant ideals $I$ of
  a row-finite graph $C^*$-algebra $C^*(E)=C^*(s_e, p_v)$ is in
  bijective correspondence with the set of
   hereditary saturated vertex subsets  $H$ in
   such a way that $I$ is the ideal generated by the projections
   $p_v$, $v\in H$. 
   By an ideal we always mean a closed two-sided one, and 
   more general description on the gauge-invariant ideal structure of an arbitrary 
   graph $C^*$-algebra is obtained in \cite{BHRS}.
   Also, for the class of ultragraph $C^*$-algebras (\cite{To1}) which contains all
  graph $C^*$-algebras (see \cite{KPR, KPRR, BHRS, BPRS} among others)
  and Exel-Laca algebras (\cite{EL}),
  the structure of gauge-invariant ideals of a ultragraph
  $C^*$-algebra was described via one to one correspondence with the set of
  admissible pairs of the ultra graph (\cite{KMST})
  using the results known for the $C^*$-algebras of
  topological graphs and topological quivers (\cite{Ka,MT}).

   The main purpose of this paper is to analyze the structure of
   gauge-invariant ideals of a labelled graph $C^*$-algebra $C^*(E,\CL,\CB)$
   when $E$ has no sinks and $(E,\CL,\CB)$ is a set-finite, receiver set-finite and
   weakly left-resolving labelled space
   such that $\CB$ is closed under taking relative complement.
  One might expect that a one to one correspondence like
  the correspondence mentioned above for graph $C^*$-algebras
  could be easily established
  by similar arguments used in the proofs for graph $C^*$-algebras
  as done in \cite{BPRS}.
  But an essential difficulty lies in the situation  that the quotient algebra
  $C^*(E,\CL,\CB)/I$ by a gauge-invariant ideal $I$ is not known
  to be realized as a labelled graph $C^*$-algebra again.
  So we introduce a notion of quotient labelled space
  $(E,\CL,\QCB_\textsc{R})$ which is similar to a labelled space
  but with the equivalence classes $\QCB_\textsc{R}$ of an
  equivalence relation $\sim_\textsc{R}$ on $\CB$
  in place of  $\CB$ of a labelled space $(E,\CL,\CB)$.
  Then in Theorem~\ref{theorem-quotientspace} 
  we associate a universal $C^*$-algebra $C^*(E,\CL,\QCB_\textsc{R})$
  to a quotient labelled space 
  and   prove that 
  every $C^*$-algebra $C^*(E,\CL,\QCB_\textsc{R})$ of a quotient
  labelled space is isomorphic to a quotient algebra
  $C^*(E,\CL,\CB)/I$ by a gauge-invariant ideal $I$ of $C^*(E,\CL,\CB)$
  in Corollary~\ref{cor-isomorphism} which follows from
  the gauge-invariant uniqueness theorem (Theorem~\ref{thm-uniqueness})
  for the $C^*$-algebras of quotient labelled spaces.   
  Moreover, it is obtained
  that if $I$ is a gauge-invariant ideal of $C^*(E,\CL,\CB)$,
  the quotient algebra $C^*(E,\CL,\CB)/I$ is isomorphic to
  a $C^*$-algebra $C^*(E,\CL,\QCB_\textsc{R})$ associated to 
  certain quotient labelled space.   
  We then apply these  isomorphism  relations  to obtain
  the main result (Theorem~\ref{theorem-main}) that
  there exists a one to one correspondence
  between the set of hereditary saturated subsets $H$
  (which we shall define, in Section 3) of $\CB$
  and the set of gauge-invariant ideals $I_H$ of $C^*(E,\CL,\CB)$.

  Returning to the labelled spaces $(E,\CL,\bE)$ and
  the simplicity of $C^*(E,\CL,\bE)$ in Section 6, we
  consider a labelled graph $(E,\CL)$ such that
  for each $v\in E^0$, a generalized vertex
  $[v]_l$ is a finite set for some $l$.
  For the merged labelled space $(F,\CL_F,\bF)$ (Definition~\ref{def-mergedgraph})
  of $(E,\CL,\bE)$, we show that $\bF$ has the property that
  every set of single vertex belongs to $\bF$ and
  $C^*(E,\CL,\bE)\cong C^*(F,\CL_F,\bF)$ (Theorem~\ref{theorem-merged-iso}).
  Moreover it is shown that
  $(F,\CL_F,\bF)$ is strongly cofinal (respectively, disagreeable)
  if and only if $(E,\CL,\bE)$ is
  strongly cofinal (respectively, disagreeable) (Theorem~\ref{theorem-cofinal}).
  This then proves that if  $(E,\CL,\bE)$ is a labelled space such that
  for each $v\in E^0$, a generalized vertex
  $[v]_l$ is finite for some $l$,  then
  $C^*(E,\CL,\bE)$ is simple if and only if
  $(E,\CL,\bE)$ is
  strongly cofinal and disagreeable (Corollary~\ref{corollary}).

\vskip 1pc

\section{Labelled  spaces and their $C^*$-algebras}

\vskip 1pc

We use notational conventions of \cite{KPR} for graphs and
graph $C^*$-algebras and of \cite{BP2} for
labelled spaces and their $C^*$-algebras.
A  {\it directed graph} is a quadruple $E=(E^0,E^1,r,s)$
consisting of a countable set of
vertices  $E^0$, a countable set of edges $E^1$,
and the range, source maps $r_E$, $s_E: E^1\to E^0$
 (we often write $r$ and $s$ for $r_E$ and $s_E$, respectively).
 By $E^n$ we denote the set of all finite paths $\ld=\ld_1\cdots \ld_n$
 of {\it length} $n$ ($|\ld|=n$),
($\ld_{i}\in E^1,\ r(\ld_{i})=s(\ld_{i+1}), 1\leq i\leq n-1$)
and use notations $E^{\leq n}:=\cup_{i=1}^n E^i$ and
$E^{\geq n}:=\cup_{i=n}^\infty E^i$.
The maps $r$ and $s$ naturally extend to $E^{\geq 1}$.
If a sequence of edges $\ld_i\in E^1(i\geq 1)$  satisfies  $r(\ld_{i})=s(\ld_{i+1})$,
one obtains an infinite path $\ld_1\ld_2\ld_3\cdots $
with the source $s(\ld_1\ld_2\ld_3\cdots):=s(\ld_1)$.
$E^\infty$ denotes the set of all infinite paths.

 A {\it labelled graph} $(E,\CL)$ over a countable alphabet  $\CA$
 consists of a directed graph $E$ and
 a {\it labelling map} $\CL:E^1\to \CA$.
 We assume that $\CL$ is onto.
 Let $\CA^*$ and $\CA^\infty$ be the sets of all finite sequences
 (of length greater than or equal to $1$)
 and infinite sequences, respectively.
 Then $\CL(\ld):=\CL(\ld_1)\cdots \CL(\ld_n)\in \CA^*$ if
  $\ld=\ld_1\cdots \ld_n\in E^n$, and
  $\CL(\dt):=\CL(\dt_1)\CL(\dt_2)\cdots\in \CL(E^\infty)\subset\CA^\infty$
  if $\dt=\dt_1\dt_2\cdots \in E^\infty$.
We use notation $\CL^*(E):=\CL(E^{\geq 1})$.
 The {\it range} $r(\af)$ and {\it source} $s(\af)$
 of a labelled path $\af\in \CL^*(E)$ are
 subsets of $E^0$ defined by
\begin{align*}
r(\af) &=\{r(\ld) \,:\, \ld\in E^{\geq 1},\,\CL(\ld)=\af\},\\
 s(\af) &=\{s(\ld) \,:\, \ld\in E^{\geq 1},\, \CL(\ld)=\af\}.
\end{align*}
 The {\it relative range of $\af\in \CL^*(E)$
 with respect to $A\subset 2^{E^0}$} is defined to be
$$
 r(A,\af)=\{r(\ld)\,:\, \ld\in E^{\geq 1},\ \CL(\ld)=\af,\ s(\ld)\in A\}.
$$
 If $\CB\subset 2^{E^0}$ is a collection of subsets of $E^0$ such that
 $r(A,\af)\in \CB$ whenever $A\in \CB$ and $\af\in \CL^*(E)$,
 $\CB$ is said to be
 {\it closed under relative ranges} for $(E,\CL)$.
 We call $\CB$ an {\it accommodating set} for $(E,\CL)$
 if it is closed under relative ranges,
 finite intersections and unions and
 contains $r(\af)$ for all $\af\in \CL^*(E)$.
  If $\CB$ is accommodating for $(E,\CL)$, the triple $(E,\CL,\CB)$ is called
 a {\it labelled space}.
 A labelled space $(E,\CL,\CB)$ is {\it weakly left-resolving} if
 $$r(A,\af)\cap r(B,\af)=r(A\cap B,\af)$$
 for all $A,B\in \CB$ and  $\af\in \CL^*(E)$.

 For $A,B\in 2^{E^0}$ and $n\geq 1$, let
 $$ AE^n =\{\ld\in E^n\,:\, s(\ld)\in A\},\ \
  E^nB=\{\ld\in E^n\,:\, r(\ld)\in B\},$$
 and
 $AE^nB=AE^n\cap E^nB$.
 We write $E^n v$ for $E^n\{v\}$ and $vE^n$ for $\{v\}E^n$,
 and will use notations $AE^{\geq k}$ and $vE^\infty$
 which should have obvious meaning.
 A labelled space $(E,\CL,\CB)$ is said to be {\it set-finite}
 ({\it receiver set-finite}, respectively) if for every $A\in \CB$
 the set  $\CL(AE^1)$ ($\CL(E^1A)$, respectively) finite.

\vskip 1pc
\noindent
{\bf Assumption 1.}\label{assumptions-graph}
 Throughout this paper, we assume that $E$ has no sinks, that is
$|s^{-1}(v)|>0$ for all $v\in E^0$.

\vskip 1pc

\begin{dfn} (\cite[Definition 4.1]{BP1})
\label{definition-representation}
Let $(E,\CL,\CB)$ be a weakly left-resolving labelled space.
A {\it representation} of $(E,\CL,\CB)$
consists of projections $\{p_A\,:\, A\in \CB\}$ and
partial isometries
$\{s_a\,:\, a\in \CA\}$ such that for $A, B\in \CB$ and $a, b\in \CA$,
\begin{enumerate}
\item[(i)]  $p_{\emptyset}=0$, $p_Ap_B=p_{A\cap B}$, and
$p_{A\cup B}=p_A+p_B-p_{A\cap B}$,
\item[(ii)] $p_A s_a=s_a p_{r(A,a)}$,
\item[(iii)] $s_a^*s_a=p_{r(a)}$ and $s_a^* s_b=0$ unless $a=b$,
\item[(iv)] for $A\in \CB$, if  $\CL(AE^1)$ is finite and non-empty, then
$$p_A=\sum_{a\in \CL(AE^1)} s_a p_{r(A,a)}s_a^*.$$
\end{enumerate}

\vskip 1pc

\begin{remark}\label{elements}
It is known \cite[Theorem 4.5]{BP1} that
if $(E,\CL,\CB)$ is a weakly left-resolving labelled space,
there exists a $C^*$-algebra $C^*(E,\CL,\CB)$
generated by a universal representation
$\{s_a,p_A\}$ of $(E,\CL,\CB)$.
In this case, we simply write $C^*(E,\CL,\CB)=C^*(s_a,p_A)$
and call $C^*(E,\CL,\CB)$ the {\it labelled graph $C^*$-algebra} of
a labelled space $(E,\CL,\CB)$.
 Furthermore,  $s_a\neq 0$ and $p_A\neq 0$ for $a\in \CA$
and  $A\in \CB$, $A\neq \emptyset$.
Note also that
$s_\af p_A s_\bt^*\neq 0$ if and only if $A\subset r(\af)\cap r(\bt)$.
If we assume that $(E,\CL,\CB)$ is set-finite,
by \cite[Lemma 4.4]{BP1} and
Definition~\ref{definition-representation}(iv)
it follows that
\begin{eqnarray}\label{eqn-projectionsum}
p_A=\sum_{\sm\in \CL(AE^n)} s_\sm p_{r(A,\sm)}s_\sm^*
\ \text{ for }   A\in \CB,\  n\geq 1
\end{eqnarray}
and
\begin{eqnarray}\label{eqn-elements}
C^*(E,\CL,\CB)=\overline{span}\{s_\af p_A s_\bt^*\,:\,
\af,\,\bt\in \CL^*(E),\ A\in \CB\}.
\end{eqnarray}
\end{remark}

\vskip 1pc
By universal property of  $C^*(E,\CL,\CB)=C^*(s_a, p_A)$,
there exists a strongly continuous action
$\gm:\mathbb T\to Aut(C^*(E,\CL,\CB))$,  called
the {\it gauge action}, such that
$\gm_z(s_a)=zs_a$ and  $\gm_z(p_A)=p_A$.
We have the following gauge-invariant uniqueness theorem for
labelled graph $C^*$-algebras $C^*(E,\CL,\CB)$.

\vskip 1pc

\begin{thm}\label{theorem-gaugeinvariant}{\rm (\cite[Theorem 5.3]{BP1})}
Let $(E,\CL,\CB)$ be a weakly left-resolving labelled space and let
$\{S_a,P_A\}$ be a representation of $(E,\CL,\CB)$ on Hilbert space.
Take $\pi_{S,P}$ to be the representation of $C^*(E,\CL,\CB)$ satisfying
$\pi_{S,P}(s_a)=S_a$ and $\pi_{S,P}(p_A)=P_A$.
Suppose that each $P_A$ is non-zero whenever $A\neq \emptyset$, and that
there is a strongly continuous action $\bt$ of $\mathbb T$ on $C^*(S_a,P_A)$
such that for all $z\in\mathbb T$,
$\bt_z\circ \pi_{S,P}=\pi_{S,P}\circ \gm_z$. Then $\pi_{S,P}$ is faithful.
\end{thm}

\vskip 1pc

For $v,w\in E^0$, we write $v\sim_l w$ if $\CL(E^{\leq l} v)=\CL(E^{\leq l} w)$
as in \cite{BP2}.
Then  $ \sim_l $ is an equivalence relation on $E^0$.
The equivalence class $[v]_l$ of $v$  is called a {\it generalized vertex}.
Let $\Om_l(E):=E^0/\thicksim_l$.
   For $k>l$ and $v\in E^0$,  $[v]_k\subset [v]_l$ is obvious and
   $[v]_l=\cup_{i=1}^m [v_i]_{l+1}$
   for some vertices  $v_1, \dots, v_m\in [v]_l$ (\cite[Proposition 2.4]{BP2}).

\vskip 1pc
\noindent
{\bf Assumption 2.}\label{assumptions-graph}
 From now on we assume that our labelled space $(E,\CL,\CB)$
 is set-finite and receiver set-finite for any accommodating set $\CB$.

\vskip 1pc

Let $\CEa$ be the smallest accommodating set for $(E,\CL)$.
Then  $\CEa$ consists of the sets of the form
$\cup_{k=1}^m\cap_{i=1}^n  r(\bt_{i,k})$, $ \bt_{i,k}\in \CL^*(E)$,
as mentioned in \cite[Remark 2.1]{BP2}, and
is contained in every accommodating set $\CB$ for $(E,\CL)$.
 Let $\bE$ be the smallest one among the accommodating sets $\CB$ for $(E,\CL)$
such that  $A\setminus B\in \CB$ whenever $A,\,B\in \CB$.
Then $\bE$ contains all generalized vertices $[v]_l$ since
every $[v]_l$ is the relative complement of sets in $\CEa$,
more precisely,
$[v]_l=X_l(v)\setminus r(Y_l(v))$,  where
$X_l(v):=\cap_{\af\in \CL(E^{\leq l} v)} r(\af)\ \text{ and }\
 Y_l(v):=\cup_{w\in X_l(v)} \CL(E^{\leq l} w)\setminus \CL(E^{\leq l} v)$
 (\cite[Proposition 2.4]{BP2}).
Moreover if $\bE$ is weakly left-resolving then
\begin{eqnarray}\label{eqn-bE}
\bE=\{ \cup_{i=1}^n[v_i]_{l_i}: v_i\in E^0, \ l_i\geq 1,\ n\geq 1\}
\end{eqnarray}
(\cite[Proposition  3.4]{JK}).

\vskip 1pc

Let $\CB_1$ and $\CB_2$ be two accommodating sets
for  $(E,\CL)$ such that $\CB_1\subset \CB_2$.
If $C^*(E,\CL,\CB_1)=C^*(t_a, q_A)$ and  $C^*(E,\CL,\CB_2)=C^*(s_a, p_B)$,
since $\{s_a,\, p_A: a\in \CA,\ A\in \CB_1\}$ is a representation of
$(E,\CL,\CB_1)$,  by the universal property of  $C^*(E,\CL,\CB_1)$
there exists a $*$-homomorphism
$\iota: C^*(E,\CL,\CB_1)\to C^*(E,\CL,\CB_2)$ such that
$\iota(t_a)=s_a$ and $q_A=p_A$ for $a\in \CA$, $A\in \CB_1$.
Let $\af$ and $\bt $ be the gauge actions of $\mathbb T$
 on $C^*(E,\CL,\CB_1)$ and $C^*(E,\CL,\CB_2)$, respectively.
 Then $\iota\circ\af_z =\bt_z\circ \iota$ for $z\in \mathbb T$ and
 $\iota(q_A)=p_A\neq 0$ for $A\in \CB_1$, hence
 by Theorem~\ref{theorem-gaugeinvariant}
 we have the following proposition.

\vskip 1pc

\begin{prop}\label{proposition-B1andB2}
Let $\CB_1\subset \CB_2$ be two accommodating sets for a labelled graph
$(E,\CL)$ such that
$(E,\CL,\CB_i)$ is weakly left-resolving for $i=1,2$.
If $C^*(E,\CL,\CB_1)=C^*(t_a, q_A)$ and  $C^*(E,\CL,\CB_2)=C^*(s_a, p_B)$,
the homomorphism $\iota:C^*(E,\CL,\CB_1)\to C^*(E,\CL,\CB_2)$ such that
 $\iota(t_a)= s_a$ and  $\iota(q_A)= p_A$  is injective.
\end{prop}

\vskip 1pc

\begin{cor}\label{cor-isomorphism}
Let $(E,\CL,\CEa)$ and $(E,\CL,\bE)$ be weakly
left-resolving labelled spaces. Then
$C^*(E,\CL,\CEa)\cong C^*(E,\CL,\bE)$.
\end{cor}
\begin{proof}
Let $C^*(E,\CL,\CEa)=C^*(s_a, p_A)$ and $C^*(E,\CL,\bE)=C^*(t_a, q_B)$,
$a\in \CA$, $A\in \CEa$, $B\in \bE$.
Then the map $\iota: C^*(E,\CL,\CEa)\to C^*(E,\CL,\bE)$
such that
$\iota(s_a)=t_a$ and $\iota(p_A)=q_A$, $A\in \CEa$, is an isomorphism
by Proposition~\ref{proposition-B1andB2}.
For any $[v]_l\in \bE$, there are two sets $A,\, B\in \CEa$ such that
$[v]_l=A\setminus B$.
Since $A=(A\setminus B)\cup (A\cap B)$ and $A\setminus B,\, A\cap B\in \bE$,
we have $q_A=q_{A\setminus B}+q_{A\cap B}$ and so
$$q_{[v]_l}= q_{A\setminus B}=q_A-q_{A\cap B}
=\iota(p_A-p_{A\cap B})\in \iota(C^*(E,\CL,\CEa)).$$
Hence $\iota$ is surjective by (\ref{eqn-bE}).
\end{proof}

\vskip 1pc

\section{Quotient labelled spaces and their $C^*$-algebras}

\vskip 1pc
\noindent
{\bf Assumption 3.} For the rest of the paper, we
assume that every labelled space $(E,\CL,\CB)$ is weakly left-resolving
and $\CB$ is closed under taking relative complement.

\vskip 1pc

\begin{dfn}
 Let $(E,\CL,\CB)$ be a  labelled space
  and $\sim_\textsc{R}$ an equivalence relation on $\CB$.
Denote the equivalence class of $A\in \CB$ by $[A]$
(or $[A]_\textsc{R}$ in case we need to specify the relation $\sim_\textsc{R}$)
and let
 $$ \CA_\textsc{R}:=\{a\in \CA:  [r(a)]\neq[\emptyset]\}.$$
If the following operations $\cup$,   $\cap$, and $\setminus$,
$$[A]\cup [B]:=[A\cup B],\  \ [A]\cap [B]:=[A\cap B],\ \ [A]\setminus [B]=[A\setminus B]$$
 are well-defined on the equivalence classes
 $\QCB_\textsc{R}:=\{[A]\,:\, A\in \CB\}$, and
if the {\it relative range},
$$r([A],\af):=[r(A,\af)],$$
is well-defined for $[A]\in \QCB_\textsc{R}$, $\af\in\CL^*(E)\cap (\CA_\textsc{R})^*$
so that  $r([A],\af) =[\emptyset]$ for all
 $\af\in\CL^*(E)\cap (\CA_\textsc{R})^*$ implies $[A]=[\emptyset]$,
 we call a triple $(E,\CL, \QCB_\textsc{R})$
 a {\it quotient labelled space} of $(E,\CL,\CB)$.
 Note  that $r([\emptyset],\af)=[\emptyset]$
 and that $[A]\setminus [B]=[\emptyset]$ whenever $[A]=[B]$.
 We say that $(E,\CL, \QCB_\textsc{R})$ is {\it weakly left-resolving}
 if $r([A],\af)\cap r([B],\af)=r([A]\cap [B],\af)$ holds for all
 $[A],[B]\in \QCB_\textsc{R}$
 and $\af\in \CL^*(E)\cap (\CA_\textsc{R})^*$.
\end{dfn}

\vskip 1pc
\noindent
  A labelled space itself is a quotient labelled space
  with the relation of equality and $\CA_\textsc{R}=\CA$.
 For a nontrivial and important example of quotient labelled spaces,
 see Proposition~\ref{prop-quotientspace} below.

In a similar way to Definition~\ref{definition-representation},
we define  a representation of a quotient labelled space as follows.
\vskip 1pc

\begin{dfn}\label{definition-qrepresentaion}
Let $(E,\CL,\QCB_\textsc{R})$ be a weakly left-resolving
quotient labelled space of a  labelled space $(E,\CL,\CB)$.
A {\it representation} of $(E, \CL, \QCB_\textsc{R})$ consists of
projections  $\{p_{[A]} \,:\, [A] \in \QCB_\textsc{R}\}$ and
partial isometries $\{s_a \,:\, a \in \CA_\textsc{R} \}$
subject to the relations:
\begin{enumerate}
\item $p_{[\emptyset]} =0$,  $p_{[A]} p_{[B]} = p_{[A] \cap [B]}$,
and $p_{[A]\cup [B]} = p_{[A]} + p_{[B]} - p_{[A]\cap [B]}$,
\item  $p_{[A]} s_a = s_a p_{r([A],a)}$,
\item $s_a^* s_a = p_{[r(a)]}$ and $s_a^* s_b=0$
unless $a=b$,
\item for $[A]\in \QCB_\textsc{R}$,
if $\CL([A]E^1)\cap\CA_\textsc{R}$ is non-empty, then
$$p_{[A]} =\sum_{a\in \CL([A]E^1)\cap\CA_\textsc{R}} s_ap_{r([A],a)}s_a^*.$$
\end{enumerate}
\end{dfn}
\end{dfn}

\vskip 1pc
\noindent In the expression
$p_{[A]} =\sum_{a\in \CL([A]E^1)\cap\CA_\textsc{R}} s_ap_{r([A],a)}s_a^*$
of Definition~\ref{definition-qrepresentaion}.(4),
$s_ap_{r([A],a)}s_a^*=0$ for all but finitely many
$a\in \CL([A]E^1)\cap\CA_\textsc{R}$.
In fact, if $[A]=[A']$ and $a\in \CL(AE^1)\setminus \CL(A'E^1)$,
then from $r([A],a)=r([A'],a)$ we must have
$[r(A,a)]=[r(A',a)]=[\emptyset]$ and hence $p_{r([A],a)}=0$.
Thus $s_ap_{r([A],a)}s_a^*\neq 0$ only when $a\in \cap_{[A']=[A]}\CL(A'E^1)$,
but the set $\cap_{[A']=[A]}\CL(A'E^1)$ is finite since we
assume that $(E,\CL,\CB)$ is set-finite.

\vskip 1pc

\begin{dfn}
Let $H $ be a  subset of an accommodating set $\CB$.
$H$ is said to be {\it hereditary}
if $H$ satisfies the following:
\begin{enumerate}
\item[(i)] $r(A,\af) \in H$ for all $A  \in H , \af \in \CL^{*}(E)$,
\item[(ii)]    $A \cup B \in H$ for all $A, B \in H $,
\item[(iii)] if $A \in H, B \in \CB$ with $B \subset A$, then $ B \in H$.
\end{enumerate}
By (iii) we see that $H$ is closed under finite intersections, and
moreover  $A \backslash B \in H$ for $A \in H$ and $B\in\CB$
since $A \setminus B \subset A\in H$ and $A \backslash B \in \CB$.
 A hereditary subset $H$  of $\CB$ is called {\it saturated}
if for any $A \in \CB$, $\{ r(A, a) : a \in \CA  \} \subset H$ implies that $A \in H$.
We write $\overline{H}$ for the smallest hereditary saturated set containing $H$.
\end{dfn}

\vskip 1pc

\begin{lem}\label{lemma-HI}
Let $I$ be a nonzero ideal in $C^{*}(E,\CL,\CB)=C^*(s_a,p_A)$.
Then the set
$$H_{I} := \{ A \in \CB \,:\, p_{A} \in I \}$$ is  hereditary and
saturated. Moreover, if $I$ is gauge-invariant,
 $H_{I} \neq \{ \emptyset \}$.
\end{lem}

\begin{proof}
To show that $H_{I}$ is hereditary, let $A \in H_{I}$.
Then $p_{A}s_{a} = s_{a}p_{r(A, a)}\in I$,
so that $p_{r(A, a)} = p_{r(a)}p_{r(A, a)} = s_{a}^{*}s_{a}p_{r(A, a)} \in I$
and $r(A, a) \in H_{I}$  for all $a \in \CA$.
Also if $A, B \in H_{I}$, $p_{A \cup B} = p_{A} + p_{B} - p_{A}p_{B}$ is in $I$,
that is, $A \cup B \in H_{I}$.
If $A \in H_{I}$ and $B \in \CB$ with $B \subset A$,
then $p_{B} = p_{A\cap B}=p_{A}p_{B} \in I$ and $B\in H_{I}$.

Now let $A \in \CB$ and $r(A, a)\in H_{I}$ for all $a \in \CA$.
Then the projection $p_{A} = \sum_{a \in \CL(AE^{1})} s_{a}p_{r(A, a)}s_{a}^{*}$
belongs to $I$, that is, $A \in H_{I}$ and the hereditary set $H_{I}$ is saturated.

Finally, let $I$ be a gauge-invariant ideal of $C^{*}(E,\CL,\CB)$
with $H_{I} = \{ \emptyset \}$.
If $\pi: C^{*}(E,\CL,\CB)\to C^{*}(E,\CL,\CB) / I$ is the quotient homomorphism,
$\{\pi(p_A),\, \pi(s_a)\}$ is a representation of $C^{*}(E,\CL,\CB)$
such that the projections $\pi(p_A)$ are nonzero
for $A\neq \emptyset$ and the action  $\gm'$ on $C^{*}(E,\CL,\CB)/I$
induced by the gauge action $\gm$ on $C^{*}(E,\CL,\CB)$
satisfies that $\gm' \circ \pi = \pi \circ \gm$ (note that $\gm_{z}(I) \subset I$).
By  Theorem~\ref{theorem-gaugeinvariant}, $\pi$ is faithful and so $I =\{0\}$.
\end{proof}

\vskip 1pc

\begin{prop}\label{prop-quotientspace}
Let $I$ be a nonzero gauge-invariant ideal of $C^*(E,\CL,\CB)$.
Then the relation
$$A \sim_I B \Longleftrightarrow A \cup W = B \cup W \ \text{for some } W \in H_I$$
defines an equivalence relation $\sim_I$ on $\CB$ such that
$(E,\CL,\QCB_I)$ is a weakly left-resolving quotient labelled space
of $(E,\CL,\CB)$.
\end{prop}

\begin{proof}
Clearly $\sim_I$ is reflexive and symmetric.
It is  transitive since
\begin{align*}
&\ A \sim_I B\text{ and }  B \sim_I C\\
\Longrightarrow &\ A \cup W = B \cup W \text{ and } B \cup V = C \cup V
 \ \text{for some } W,\, V\in H_I\\
\Longrightarrow &\   A \cup ( W \cup V)= C \cup (W \cup V)\\
\Longrightarrow &\  A \sim_I C \ \text{because } W\cup V\in H_I.
\end{align*}
 To see that we have well-defined operations
 $\cup$, $\cap$ and $\setminus$ on $\QCB_I$,
 let $[A]=[A']$ and $[B]=[B']$.
Choose $W,\,V \in  H_I$ such that $A \cup W =A' \cup W$ and $B \cup V = B' \cup V$.
Then
\begin{align*}
(A \cup B)\cup (W \cup V) &= (A' \cup B')\cup (W \cup V),\\
(A \cap B)\cup (W \cup V) &=(A\cup (W \cup V)) \cap( B\cup (W \cup V))\\
&=(A'\cup (W \cup V)) \cap( B'\cup (W \cup V))\\
&= (A' \cap B')\cup (W \cup V),\\
(A\setminus B)\cup(W\cup V) &=(A'\setminus B')\cup(W\cup V).
\end{align*}
Thus
$[A]\cup[B] =[A']\cup[B']$,
$[A]\cap[B] =[A']\cap[B']$, and
$[A]\setminus [B] =[A']\setminus [B']$.

We claim that $[r(A,\af)]=[r(A',\af)]$ for $[A]=[A']$
and $\af\in \CL^*(E)\cap \CA_I^*$, where
$\CA_I =\{a\in \CA:[r(a)]\neq [\emptyset]\}= \{a\in \CA : p_{r(a)}\notin I\}$.
Let $A\cup W=A'\cup W$ for $W\in H_I$.
Then
$r(A,\af) \cup r(W,\af)=r(A\cup W, \af)
=r(A'\cup W, \af)=r(A',\af) \cup r(W,\af)$.
Since $r(W,\af)\in H_I$,  we have $[r(A,\af)]=[r(A',\af)]$
and see that the relative ranges $r([A],\af)$ are well-defined.

If $r([A],\af)=[\emptyset]$ for all $\af\in \CL^*(E)\cap \CA_I^*$, then
 $r(A,a)\in H_I$ for all $a\in \CA_I$.
 Since  $r(A,a)\in H_I$  for all $a\notin \CA_I$ and $H_I$ is saturated,
 $A\in H_I$, that is, $[A]=[\emptyset]$ follows.

  Finally, $\QCB_I$ is weakly left-resolving since
$ r( [A], \af) \cap r([ B], \af)
= [r( A, \af)] \cap [r( B, \af)]
= [r( A, \af) \cap r( B, \af)]
= [r( A\cap B,\af)]
= r( [A\cap B],\af)
= r( [A]\cap[ B],\af)$.
\end{proof}

\vskip 1pc

\begin{lem}\label{lemma-IH}
Let $H$ be a hereditary subset of $\CB$.
Then the ideal $I_{H}$ of $C^{*}(E,\CL,\CB)$
generated by the projections $\{ p_A \,:\, A \in H \}$ is gauge-invariant and
$$I_{H}= I_{\bH} =\overline{span}\{ s_{\af}p_{A}s_{\bt} \,:\, \af, \bt \in \CL^{*}(E), A \in \overline H \}.$$
\end{lem}

\begin{proof}
By Lemma~\ref{lemma-HI}, $H_{I_{H}}= \{ A \in \bE \,:\, p_{A} \in I_{H} \}$
is a hereditary saturated subset of $\CB$
and $H \subset H_{I_{H}}$.
Thus $\bH \subset H_{I_{H}}$.
It is easy to see that
$J:=\overline{span} \{ s_{\af}p_{A}s_{\bt} \,:\, \af, \bt \in \CL^{*}(E), A \in \bH \}$
is a gauge-invariant ideal of $C^{*}(E,\CL,\CB)$ such that  $J\subset I_{H}$.
But $J$ contains the generators $\{ p_{A} \,:\, A \in \bH \}$ of $I_{\bH}$.
Hence $I_{\bH}\subset J$.
\end{proof}

\vskip 1pc

Let $(E,\CL,\QCB_\textsc{R} )$  be a quotient labelled space.
As in \cite{BP1}, we set
 $$\QCB_\textsc{R}^*=(\CL^*(E)\cap (\CA_\textsc{R})^*)\cup \QCB_\textsc{R}$$ and
extend $r,s$ to $\QCB_\textsc{R}^*$ by
$r([A])=[A]$ and $s([A])=[A]$ for $[A]\in \QCB_\textsc{R}$.
Also put $s_{[A]}=p_{[A]}$ for $[A]\in \QCB_\textsc{R}$  so that
$s_\bt$ is defined for all $\bt\in \QCB_\textsc{R}^*$.
The following  lemma can be proved by the same arguments
in \cite[Lemma 4.4]{BP1}.

\vskip 1pc

\begin{lem}\label{lemma-elements}
Let $(E,\CL,\QCB_\textsc{R})$ be a weakly left-resolving quotient labelled space
and $\{s_a, p_{[A]}\}$ a representation of $(E,\CL,\QCB_\textsc{R})$.
Then any nonzero products of $s_a$, $p_{[A]}$, and $s_\bt^*$ can be
written as a finite linear combination of elements of the form
$ s_a p_{[A]} s_\bt^*$ for some $A\in \QCB_\textsc{R}$
and $\af,\bt\in \QCB_\textsc{R}^*$ with
$[A]\subset [r(\af)\cap r(\bt)]\neq [\emptyset].$
Moreover we have the following:
$$(s_\af p_{[A]} s_\bt^*)(s_\gm p_{[B]} s_\dt^*)=\left\{
                      \begin{array}{ll}
                        s_{\af\gm'}p_{r([A],\gm')\cap [B]} s_\dt^*, & \hbox{if\ } \gm=\bt\gm' \\
                        s_{\af}p_{[A]\cap r([B],\bt')} s_{\dt\bt'}^*, & \hbox{if\ } \bt=\gm\bt'\\
                        s_\af p_{[A]\cap [B]}s_\dt^*, & \hbox{if\ } \bt=\gm\\
                        0, & \hbox{otherwise.}
                      \end{array}
                    \right.
$$
\end{lem}

\vskip 1pc

\begin{thm}\label{theorem-quotientalgebra}
Let $(E,\CL,\CB)$ be a  labelled space
and $I$ be a nonzero gauge-invariant ideal of $C^*(E,\CL,\CB)$.
Then there exists a $C^*$-algebra $ C^*(E,\CL,\QCB_I)$ generated by
a universal representation  $\{t_a,p_{[A]}\}$ of $(E,\CL,\QCB_I)$.
Furthermore  $p_{[A]}\neq 0$ for  $[A] \neq [\emptyset]$
and  $t_a\neq 0$ for $a\in \CA_I$.
\end{thm}

\begin{proof}
The existence of the $C^*$-algebra $C^*(E,\CL,\QCB_I)$ with
the desired universal property can be shown by  the same argument
in the first part of the proof of \cite[Theorem 4.5]{BP1},
and here we show the second assertion of our theorem.
If $C^*(E,\CL,\CB)=C^*(s_a,p_A)$, it is easy to see that
$\{s_a+I,\ p_A +I :  a\in \CA_I, \, [A]\in \QCB_I\}$ is a representation of
$C^*(E,\CL,\QCB_I)$, hence there is a homomorphism
$\psi :   C^*(E,\CL,\QCB_I)  \to C^*(E,\CL,\CB)/I$ such that
$$\psi(t_a)=s_a+I, \qquad \psi(p_{[A]})=p_A + I.$$
If $\psi(p_{[A]})=p_A+I= I$, then $p_A \in I$ and it follows that $A \in H_I$, that is,
 $[A]=[\emptyset]$.
Thus if $[A] \neq [\emptyset]$ then $\psi(p_{[A]}) \neq  I$, and so $p_{[A]}\neq 0$.
If  $\psi(t_a)=s_a+I= I$, then  $s_a^* s_a +I= p_{r(a)}+I= I$, and so
$[r(a)]=[\emptyset]$, that is, $a\notin \CA_I$.
Thus if $a\in \CA_I$, then $\psi(t_a)\neq I$, hence $t_a\neq 0$.
\end{proof}

\vskip 1pc

The following theorem together with
Corollary~\ref{cor-isomorphism} and Theorem~\ref{theorem-main}
shows that the $C^*$-algebra $C^*(E,\CL,\QCB_\textsc{R})$
of a weakly left-resolving quotient labelled space $(E,\CL,\QCB_\textsc{R})$
is always isomorphic to a $C^*$-algebra $C^*(E, \CL, \QCB_I)$
for some gauge-invariant ideal $I$ of $C^*(E,\CL,\CB)$.

\vskip 1pc

\begin{thm}\label{theorem-quotientspace}
Let $(E,\CL,\QCB_\textsc{R})$ be a weakly left-resolving
quotient labelled space of $(E,\CL,\CB)$.
Then there exists a $C^*$-algebra $C^*(E,\CL,\QCB_\textsc{R})$
generated by a universal representation
$\{t_b, q_{[A]}\}$ of $(E,\CL,\QCB_\textsc{R})$ such that
$q_{[A]}\neq 0$ for $[A]\neq [\emptyset]$
and $t_b\neq 0$ for $b\in \CA_\textsc{R}$.
Moreover the ideal $I$ of $C^*(E,\CL,\CB)=C^*(s_a,p_A)$ generated by
the projections $p_A$, $[A]=[\emptyset]$, is gauge-invariant and
there exists a surjective homomorphism
$$\phi: C^*(E,\CL,\QCB_\textsc{R})\to C^*(E,\CL,\CB)/I$$ such that
$\phi(t_b)=s_b+I$ and $\phi(q_{[A]})=p_A+I$.
\end{thm}

\begin{proof}
One can show the existence of $C^*(E,\CL,\QCB_\textsc{R})$ with the
universal property as usual.

Let $C^*(E,\CL,\CB)=C^*(s_a,p_A)$, $a\in \CA$, $A\in \CB$,
and let $C^*(E,\CL,\QCB_R)=C^*(t_b,q_{[A]})$,
$b\in \CA_R$, $[A]\in \QCB_\textsc{R}$.
One can easily see that the ideal $I$ generated by the projections $p_A$,
$[A]=[\emptyset]$, is gauge-invariant.

 Now we show that $\CA_\textsc{R}=\CA_I$ (recall
 $\CA_\textsc{R}=\{a\in \CA: [r(a)]\neq [\emptyset]\}$ and
 $\CA_I:=\{a\in \CA: p_{r(a)}\notin I\})$.
 $\CA_I\subset \CA_\textsc{R}$ follows from the fact that
  $p_{r(a)}\in I$ whenever $[r(a)]=[\emptyset]$ by definition of $I$.
 To prove the reverse inclusion, we first show that
 \begin{eqnarray}\label{eqn-projectioninideal}
 p_A\notin I \ \text{ when }\ [A]\neq [\emptyset],
 \end{eqnarray}
 then, since $a\in \CA_\textsc{R}$ if and only if  $[r(a)]\neq [\emptyset]$,
 by  (\ref{eqn-projectioninideal}) $a\in \CA_\textsc{R}$ implies
 $p_{r(a)}\notin I$, thus $a\in \CA_I$.
 To prove (\ref{eqn-projectioninideal}),
 we suppose $[A]\neq [\emptyset]$ and $p_A\in I$.
 It is not hard to see from (\ref{eqn-elements})
 that the span of elements of the form
 $s_\af p_B s_\bt^*$, $[B]=[\emptyset]$, is dense in $I$,
 so we can find $c_i\in \mathbb{C}$, $\af_i,\, \bt_i\in \CL^*(E)$, and
 $B_i\in \CB$ with $[B_i]=[\emptyset]$
 and $B_i\subset r(\af_i)\cap r(\bt_i)$ for $i=1,\cdots, n$ such that
 $$1>\| p_A -\sum_{i=1}^n c_i s_{\af_i} p_{B_i} s_{\bt_i}^*\|.$$
  Using (\ref{eqn-projectionsum})
  we may assume that the lengths $|\af_i|$, $1\leq i\leq n$,
 are all equal to, say $l$, and write $p_A$ as a finite sum
 $$p_A=\sum_{|\gm|=l} s_\gm p_{r(A,\gm)}s_\gm^*.$$
 Since $[A]\neq [\emptyset]$, there exists a $\gm_0$ such that
 $|\gm_0|=l$ and $[r(A,\gm_0)]\neq [\emptyset]$.
 Then $r(A,\gm_0)\setminus \cup_{i=1}^n B_i\neq \emptyset$, and we have
 \begin{align*}
 1 & >\ \| p_A -\sum_{i=1}^n c_i s_{\af_i} p_{B_i} s_{\bt_i}^*\| \\
   & =\ \| \sum_{|\gm|=l} s_\gm p_{r(A,\gm)}s_\gm^* -\sum_{i=1}^n c_i s_{\af_i} p_{B_i} s_{\bt_i}^*\| \\
   & \geq\ \| s_{\gm_0}^*(\sum_{|\gm|=l} s_\gm p_{r(A,\gm)}s_\gm^*)s_{\gm_0}
   -s_{\gm_0}^*(\sum_{i=1}^n c_i s_{\af_i} p_{B_i} s_{\bt_i}^*)s_{\gm_0}\| \\
   & \ \ \ \ \ (\Lambda:=\{i: \af_i=\gm_0\})\\
   & =\ \| p_{r(A,\gm_0)}-\sum_{i\in \Lambda}c_i  p_{B_i} s_{\bt_i}^*s_{\gm_0}\|\\
   & \geq \ \|p_{r(A,\gm_0)\setminus (\cup_{i\in \Lambda}B_i)}
   \big(p_{r(A,\gm_0)} -\sum_{i\in \Lambda} c_i p_{B_i} s_{\bt_i}^*s_{\gm_0}\big)\|\\
   & =\ \|p_{r(A,\gm_0)\setminus (\cup_{i\in \Lambda}B_i)}\|\\
   & =\ 1,
 \end{align*}
 a contradiction.

 Set $T_a:=s_a+I$ and $Q_{[A]}:=p_A+I$ for
 $a\in \CA_\textsc{R}(=\CA_I)$ and $[A]\in \QCB_R$.
Here  $Q_{[A]}$ is well-defined since
$[A]=[B]$ implies $p_A-p_B\in I$.
In fact, since  $[A]=[B]$ implies $[A\setminus B]=[\emptyset]=[B\setminus A]$,
 we have $p_{A\setminus B}$, $p_{B\setminus A}\in I$, hence
 $p_A-p_B=(p_{A\setminus B}+p_{A\cap B})-(p_{B\setminus A}+p_{A\cap B})
 =p_{A\setminus B}-p_{B\setminus A} \in I$.
Note that  $Q_{[\emptyset]}=p_{\emptyset}+I=I$ and
$Q_{[A]}Q_{[B]}=(p_A+I)(p_B+I)=p_A p_B +I=p_{A\cap B}+I
= Q_{[A\cap B]}= Q_{[A]\cap [B]}$.
Similarly, $Q_{[A]\cup [B]}=  Q_{[A]} + Q_{[B]}-Q_{[A]\cap [B]}$
and (3), (4) of Definition~\ref{definition-qrepresentaion} hold,
which shows that  $\{T_a, Q_{[A]}\}$ is a representation of $(E,\CL,\QCB_\textsc{R})$.
Thus by the universal property there exists a homomorphism
$$\phi: C^*(E,\CL,\QCB_\textsc{R})\to C^*(E,\CL,\CB)/I$$ such that
$\phi(t_a)=s_a+I$ and $\phi(q_{[A]})=p_A +I$ for
$a\in \CA_\textsc{R}$ and $[A]\in\QCB_\textsc{R}$.
 Since $C^*(E,\CL,\CB)/I$ is generated by
 $\{s_a+I,\, p_A+I : a\in \CA_I,\ [A]\neq [\emptyset]\}$
 and $\CA_I=\CA_\textsc{R}$,
 it follows that $\rho$ is surjective.

If  $[A]\neq [\emptyset]$,  by  (\ref{eqn-projectioninideal})
$p_A\notin I$, hence $\phi(q_{[A]})=p_A+I\neq I$.
Thus $q_{[A]}\neq 0$.
If $b\in \CA_\textsc{R}$, namely $[r(b)]\neq [\emptyset]$,
then $\phi(t_b^* t_b)=s_b^* s_b+I=p_{r(b)}+I \neq I$
again by  (\ref{eqn-projectioninideal}).
Hence $t_b\neq 0$ in $C^*(E,\CL,\QCB_\textsc{R})$.
\end{proof}

\vskip 1pc

\begin{dfn}
We call the $C^*$-algebra $C^*(E,\CL,\QCB_\textsc{R})$
of Theorem~\ref{theorem-quotientspace}
the {\it quotient labelled graph $C^*$-algebra}.
\end{dfn}

\vskip 1pc

\section{Gauge-invariant Uniqueness Theorem for $C^*(E,\CL,\QCB)$}

\vskip 1pc

By the universal property, it follows that
every quotient labelled graph $C^*$-algebra
$C^*(E,\CL,\QCB_\textsc{R})=C^*(s_a,p_{[A]})$
admits the {\it gauge action} $\gm$ of $\mathbb{T}$ such that
$$\gm_z(s_a)=z s_a \ \text{ and }\ \gm_z(p_{[A]})=p_{[A]}$$
for $a\in \CA_\textsc{R}$ and $[A]\in \QCB_\textsc{R}$.

 The gauge-invariant uniqueness theorem for quotient labelled graph
 $C^*$-algebras can be proved by the same arguments used
 in the proof of \cite[Theorem 5.3]{BP1} for
 the $C^*$-algebras of labelled spaces, but for the convenience of readers
 we give a sketch of the proof  with some minor corrections
 to the proof of \cite[Lemma 5.2]{BP1} and \cite[Theorem 5.3]{BP1}.

\vskip 1pc

\begin{lem}\label{lemma-minimalprojection}
Let $(E,\CL,\QCB_\textsc{R})$ be a weakly left-resolving quotient labelled space
of a  labelled space $(E,\CL,\CB)$,
$\{s_a, p_{[A]}\}$ a representation of $(E,\CL,\QCB_\textsc{R})$, and
$Y=\{s_{\af_i}p_{[A_i]}s_{\bt_i}^*: i=1,\dots, N\}$ be a set of
partial isometries in $C^*(E,\CL,\QCB_\textsc{R})$ which is closed under
multiplication and taking adjoints.
Then any minimal projection of $C^*(Y)$ is equivalent to
a minimal projection $q$ in $C^*(Y)$ that is  either
\begin{enumerate}
\item[(i)] $q=s_{\af_i}p_{[A_i]} s_{\af_i}^*$ for some $1\leq i\leq N$;
\item[(ii)] $q=s_{\af_i}p_{[A_i]} s_{\af_i}^*-q'$, where
$q'=\sum_{l=1}^m s_{\af_{k(l)}}p_{[A_{k(l)}]} s_{\af_{k(l)}}^*$ and $1\leq i\leq N$;
moreover there is a nonzero $r=s_{\af_i\bt} p_{[r(A_i,\bt )]}s_{\af_i\bt}^*\in C^*(E,\CL,\QCB_\textsc{R})$
such that $q'r=0$ and $q\geq r$.
\end{enumerate}
\end{lem}

\begin{proof} A minimal projection  of the finite dimensional
$C^*$-algebra $C^*(Y)$ is unitarily equivalent to a projection of the form
$$\sum_{j=1}^n s_{\af_{i(j)}} p_{[A_{i(j)}]} s^*_{\af_{i(j)}}
-\sum_{l=1}^m s_{\af_{k(l)}} p_{[A_{k(l)}]} s^*_{\af_{k(l)}}$$
where the projections in each sum are mutually orthogonal
and for each $l$ there is a unique $j$ such that
$s_{\af_{i(j)}} p_{[A_{i(j)}]} s^*_{\af_{i(j)}}\geq
s_{\af_{k(l)}} p_{[A_{k(l)}]} s^*_{\af_{k(l)}}$.
Then the same argument of the proof of \cite[Lemma 5.2]{BP1} proves
the assertion.
\end{proof}

\vskip 1pc

\begin{thm}\label{thm-uniqueness}
Let $(E,\CL,\QCB_\textsc{R})$ be a weakly left-resolving quotient labelled space
and  $C^*(E,\CL,\QCB_\textsc{R})=C^*(s_a, p_{[A]})$.
Let $\{S_a, P_{[A]}\}$ be a representation of $(E,\CL,\QCB_\textsc{R})$
such that each $P_{[A]}\neq 0$ whenever $[A] \neq [\emptyset]$
and  $S_a\neq 0$ whenever $[r(a)]\neq [\emptyset]$.
If $\pi_{S,P}:C^*(E,\CL,\QCB_\textsc{R})\to C^*(S_a,P_{[A]})$
is the homomorphism satisfying
$\pi_{S,P}(s_a)=S_a$, $\pi_{S,P}(p_{[A]})=P_{[A]}$ and if
  there is a strongly continuous action $\bt$ of $\mathbb{T}$ on $C^*(S_a, P_{[A]})$
such that $\bt_z \circ \pi = \pi \circ \gamma_z$,
then $\pi_{S,P}$ is faithful.
\end{thm}

\begin{proof}
It is standard (for example, see the proof of \cite[Theorem 5.3]{BP1})
to show that the fixed point algebra $C^*(E,\CL,\QCB_\textsc{R})^\gm$
is equal to
$$ \overline{\rm span}\{ s_\af p_{[A]} s^*_\bt:
\af,\bt\in \CL^*(E)\cap (\CA_I)^*, \ |\af|=|\bt|, \text{ and } [
A]\in\QCB_\textsc{R} \}.$$
 Note that $C^*(E,\CL,\QCB_\textsc{R})^\gm$ is an AF algebra.
 In fact, if $Y$ is a finite subset of $C^*(E,\CL,\QCB_\textsc{R})^\gm$,
 each element $y\in Y$
 can be approximated by linear combinations of
 $s_\af p_{[A]} s^*_\bt$ with $|\af|=|\bt|$,
 hence we may assume that
 $Y=\{s_{\af_i} p_{[A_i]} s^*_{\bt_i}: |\af_i|=|\bt_i|,\ i=1,\dots,N\}$.
 Using Lemma~\ref{lemma-elements}, we may also assume that
 $Y$ is closed under multiplication and taking adjoints, so that
 $C^*(Y)=\overline{\rm span}(Y)$ is finite dimensional.
 Thus $C^*(E,\CL,\QCB_\textsc{R})^\gm$ is an AF algebra.

Now we show that
 $\pi:=\pi_{S,P}$ is faithful on  $C^*(E,\CL,\QCB_\textsc{R})^\gamma$.
 Let $\{Y_n:n\geq 1\}$ be an increasing family of finite subsets
 of  $C^*(E,\CL,\QCB_\textsc{R})^\gamma$ which are closed under multiplication
 and taking adjoints such that
 $C^*(E,\CL,\QCB_\textsc{R})^\gamma=\overline{\cup\, C^*(Y_n)}$.
 Suppose $\pi $ is not faithful on  $C^*(Y_n)$
 for some $Y_n=\{s_{\af_i} p_{[A_i]} s^*_{\bt_i}: i=1,\dots, N(n)\}$.
 Since $C^*(Y_n)$ is finite dimensional, the kernal of $\pi|_{C^*(Y_n)}$
 has a minimal projection.
 By Lemma~\ref{lemma-minimalprojection}, each minimal projection
 in the kernal of $\pi|_{C^*(Y_n)}$ is unitarily equivalent
 to a projection which is either $q=s_{\af_i} p_{[A_i]} s^*_{\bt_i}$ $(1\leq i\leq N(n))$
 or $q=s_{\af_i} p_{[A_i]} s^*_{\bt_i}-q'$,
 $q'=\sum_{k=1}^m s_{\af_{i(k)}} p_{[A_{i(k)}]} s^*_{\bt_{i(k)}}$ $(1\leq i\leq N(n))$.
 As in the proof of Theorem~\ref{theorem-gaugeinvariant}(\cite[Theorem 5.3]{BP1}),
 one obtains $\pi (q)\neq 0$  in each case.
 Then
 $\pi (uqu^*)\neq 0$ for any unitaty  $u\in C^*(Y_n)$,
 namely $\pi$ maps every minimal projection to a nonzero element and
 hence is faithful on $C^*(E,\CL,\QCB_\textsc{R})^\gm$.

  Therefore we conclude that $\pi$ is faithful
  by \cite[Lemma 2.2]{BKR} since the following  holds:
 For $a \in C^*(E,\CL,\QCB_\textsc{R})$,
  $$ \| \pi \big( \int_{\mathbb{T}} \gamma_z(a)dz \big)  \| \leq
  \int_{\mathbb{T}} \|\pi(\gm_z(a)\|dz=
   \int_{\mathbb{T}} \|\bt_z(\pi(a))\|dz= \| \pi(a)  \|.
  $$
\end{proof}

\vskip 1pc

\begin{cor}\label{cor-isomorphism}
Let $(E,\CL,\QCB_\textsc{R})$ be a weakly left-resolving quotient labelled space
of $(E,\CL,\CB)$ and let $C^*(E,\CL,\QCB_\textsc{R})=C^*(t_b, q_{[A]})$.
If $C^*(E,\CL,\CB)=C^*(s_a,p_A)$ and
$I$ is the ideal generated by the projections $q_{[A]}$, $[A]=[\emptyset]$,
there is a surjective isomorphism
$$\phi: C^*(E,\CL,\QCB_\textsc{R})\ \to\ C^*(E,\CL,\CB)/I$$
such that
$\phi(t_b)=s_b+I$, $\phi(q_{[A]})=p_A+I$ for $b\in \CA_\textsc{R}(=\CA_I)$,
$[A]\in \QCB_\textsc{R}$.
\end{cor}

\begin{proof}
By Theorem~\ref{theorem-quotientspace}, we have a surjective homomorphism
$\phi$ with the desired properties except injectivity.
But it is injective by the gauge-invariant uniqueness theorem since
$\phi(q_{[A]})=p_A+I$ is nonzero for all $[A]\in \QCB_\textsc{R}$,
$[A]\neq [\emptyset]$,
and $\bt_z \circ \phi = \phi \circ \gamma_z$ for $z\in \mathbb{T}$,
where $\bt$ is the gauge action  on $C^*(E,\CL,\CB)/I$ induced by the
gauge action on $C^*(E,\CL,\CB)$ and $\gm$ is the gauge action on
$C^*(E,\CL,\QCB_\textsc{R})$.
\end{proof}

\vskip 1pc

\section{Gauge-invariant ideals of labelled graph $C^*$-algebras}

\vskip 1pc
Recall that for a labelled space $(E,\CL,\CB)$  and
a hereditary saturated subset $H$ of $\CB$,
$I_H$ denotes the ideal of $C^*(E,\CL,\CB)$ generated by
the projections $p_A$, $A\in H$ (see Lemma~\ref{lemma-IH}).

\vskip 1pc

\begin{lem}\label{lemma-injection}
Let $(E,\CL,\CB)$ be a labelled space.
Then the map $H \longmapsto I_{H}$ is an inclusion preserving injection
from the set of nonempty hereditary saturated subsets of $\CB$ into
the set of nonzero gauge-invariant ideals of $C^*(E,\CL,\CB)$.
\end{lem}

\begin{proof}
Clearly the map is inclusion preserving.
For injectivity, we show that
the composition of $H\mapsto I_{H}$ and
$I \mapsto  H_{I}$ is the identity on
the set of hereditary saturated subsets of $\CB$,
that is, we show that $H_{I_H}=H$.
From the easy fact that $I_{H_J}\subset J$ holds for
any ideal $J$, we see with $J=I_H$ that  $I_{H_{I_H}}\subset I_H$,
which then shows $H_{I_H}\subset H$.
Since $H\subset H_{I_H}$ is rather obvious, we have $H_{I_H}=H$.
\end{proof}

\vskip 1pc

\begin{thm}\label{theorem-main}
Let $I$ be a nonzero gauge-invariant ideal of $C^*(E,\CL,\CB)$.
Then there exists an isomorphism of $C^*(E,\CL,\QCB_I)$ onto
the quotient algebra $C^*(E,\CL,\CB)/I$ and
$I=I_H$, where $H$ is the hereditary saturated subset
consisting of $A\in \CB$ with $p_A\in I$.
Moreover
the map $H\mapsto I_H$ gives an inclusion preserving
bijection between the nonempty hereditary saturated subsets of  $\CB$
and the nonzero gauge-invariant ideals of $C^*(E,\CL,\CB)$.
\end{thm}

\begin{proof}
Let $I$ be a nonzero gauge-invariant ideal of $C^*(E,\CL,\CB)=C^*(s_a,p_v)$
and $\sim_I$ the equivalence relation on $\CB$ defined by
$$A\sim_I B \Longleftrightarrow A\cup W =B\cup W \ \text{ for
some }\ W\in H,$$
where $H:=H_I=\{A\in \CB: p_A\in I\}$ is
the hereditary saturated subset of $\CB$ (Lemma~\ref{lemma-HI}).
Then $(E,\CL,\QCB_I)$
is a weakly left-resolving quotient labelled space
by Proposition~\ref{prop-quotientspace} and
we see from the proof of Theorem~\ref{theorem-quotientalgebra}
that  there exists a surjective homomorphism
$$\psi : C^*(E,\CL,\QCB_I) \to C^*(E,\CL,\CB)/I$$
such that
$\psi(t_b)=s_b+I$, $\psi(q_{[A]})=p_A + I$
for $b\in \CA_{I}$, $[A]\in\QCB_I$.
Moreover $p_A + I\neq I$ and $ s_b+I\neq I$.
By applying the gauge-invariant uniqueness theorem (Theorem~\ref{thm-uniqueness}),
we see that $\psi$ is an isomorphism, which proves the first assertion.
 On the other hand, the ideal $I_H(\subset I)$ of $C^*(E,\CL,\CB)$
 generated by the projections $p_A\in I$ is gauge-invariant and
 $\CA_I=\CA_{I_H}$ since
$$[A]=[\emptyset] \Longleftrightarrow p_A\in I_H
\Longleftrightarrow p_A\in I.$$
By Corollary~\ref{cor-isomorphism}
with $\sim_I$ in place of $\sim_\textsc{R}$,
we have a surjective isomorphism
$$\phi : C^*(E,\CL,\QCB_I) \to C^*(E,\CL,\CB)/I_H$$
such that
$\phi(q_{[A]})=p_A + I_H, \ \phi(t_b)=s_b+I_H$
for $b\in \CA_I$ and $[A]\in \QCB_I$,
where $C^*(E,\CL,\QCB_I)=C^*(t_b, q_{[A]})$.
Then the composition of $\psi$ and $\phi^{-1}$,
$$\psi \circ \phi^{-1}: C^*(E,\CL,\CB)/I_H \to C^*(E,\CL,\CB)/I$$
is an isomorphism such that
$$\psi \circ \phi^{-1}(p_A + I_H)=p_A + I,\ \psi \circ \phi^{-1}(s_a+I_H)=s_a+I,$$
which shows $I=I_H$.
Finally Lemma~\ref{lemma-injection} completes the proof.
\end{proof}

\vskip 1pc

\begin{ex}
If $(E, \CL)$ is the following labelled graph

\vskip 1pc

\hskip .5pc
\xy
/r0.38pc/:(-33,0)*+{\cdots};
(33,0)*+{\cdots\, ,};
(-30,0)*+{\bullet}="V-3";
(-20,0)*+{\bullet}="V-2";
(-10,0)*+{\bullet}="V-1";
(0,0)*+{\bullet}="V0";
(10,0)*+{\bullet}="V1";
(20,0)*+{\bullet}="V2";
(30,0)*+{\bullet}="V3";
 "V-3";"V-2"**\crv{(-30,0)&(-20,0)};
 ?>*\dir{>}\POS?(.5)*+!D{};
 "V-2";"V-1"**\crv{(-20,0)&(-10,0)};
 ?>*\dir{>}\POS?(.5)*+!D{};
 "V-1";"V0"**\crv{(-10,0)&(0,0)};
 ?>*\dir{>}\POS?(.5)*+!D{};
 "V0";"V1"**\crv{(0,0)&(10,0)};
 ?>*\dir{>}\POS?(.5)*+!D{};
 "V1";"V2"**\crv{(10,0)&(20,0)};
 ?>*\dir{>}\POS?(.5)*+!D{};
 "V2";"V3"**\crv{(20,0)&(30,0)};
 ?>*\dir{>}\POS?(.5)*+!D{};
 "V-2";"V-3"**\crv{(-20,0)&(-25,-6)&(-30,0)};
 ?>*\dir{>}\POS?(.5)*+!D{};
 "V-1";"V-2"**\crv{(-10,0)&(-15,-6)&(-20,0)};
 ?>*\dir{>}\POS?(.5)*+!D{};
 "V0";"V-1"**\crv{(0,0)&(-5,-6)&(-10,0)};
 ?>*\dir{>}\POS?(.5)*+!D{};
 "V1";"V0"**\crv{(10,0)&(5,-6)&(0,0)};
 ?>*\dir{>}\POS?(.5)*+!D{};
 "V2";"V1"**\crv{(20,0)&(15,-6)&(10,0)};
 ?>*\dir{>}\POS?(.5)*+!D{};
 "V3";"V2"**\crv{(30,0)&(25,-6)&(20,0)};
 ?>*\dir{>}\POS?(.5)*+!D{};
 "V0";"V0"**\crv{(0,0)&(-4,4)&(0,8)&(4,4)&(0,0)};
 ?>*\dir{>}\POS?(.5)*+!D{};
 (-25,2)*+{b};(-15,2)*+{b};(-5,2)*+{b};(5,2)*+{b};(15,2)*+{b};(25,2)*+{b};
 (-25,-4)*+{c};(-15,-4)*+{c};(-5,-4)*+{c};(5,-4)*+{c};(15,-4)*+{c};(25,-4)*+{c};
 (0,8)*+{a};(0.1,-3)*+{v_0};(10.1,-3)*+{v_1};
 (-9.9,-3)*+{v_{-1}};
 (-19.9,-3)*+{v_{-2}};
 (20.1,-3)*+{v_{2}};
\endxy

\vskip 1pc
\noindent
then $C^*(E,\CL,\CEa)\cong C^*(E,\CL,\bE)$ by Corollary~\ref{cor-isomorphism}
while
\begin{align*}
& \CEa = \{E^0\}\cup \{A\subset E^0\,:\, A \ {\rm is\  finite }\},\\
 & \bE  =\CEa\cup\{ A \subset E^{0} \,:\,  E^0\setminus A\  {\rm  is \ finite} \}.
\end{align*}
Let $I$ be the gauge-invariant ideal of $C^*(E,\CL,\bE)$ corresponding to
the hereditary saturated set $H= \{A\subset E^0\,:\, A \ {\rm is\  finite }\}$.
Then $[\bE]_I=\{[E^0], [\emptyset]\}$ and $\CA_I=\{b,c\}$.
Let  $C^*(E,\CL,[\bE]_I)=C^*(p_{[E^0]},s_b,s_c)$.
Since $s_b^*s_b=p_{[r(b)]}=p_{[E^0]}=s_c^*s_c$,
$s_b^*s_c=0$,   $p_{[E^0]}s_b=s_b p_{r([E^0],b)}=s_bp_{[E^0]}$,
and similarly $p_{[E^0]}s_c=s_c p_{[E^0]}$,
$C^*(E,\CL,[\bE]_I)$ is the universal $C^*$-algebra
generated by two isometries with orthogonal ranges with the unit $p_{[E^0]}$.
Therefore $C^*(E,\CL,[\bE]_I)$ is isomorphic to the Cuntz algebra $\CO_2$
and by Theorem~\ref{theorem-main}, we have $C^*(E,\CL, \bE)/I \cong \CO_2$.
(For the ideal $I$, see \cite[Remark 3.7]{JK}.)
\end{ex}

\vskip 1pc

\section{$C^*$-algebras of merged labelled graphs}

In this section we show that
given a labelled space $(E,\CL_E,\bE)$ such that for every $v\in E^0$,
 $[v]_l$ is finite for some $l\geq 1$ (this obviously holds if $E^0$ is finite),
there exists a labelled space  $(F,\CL_F,\bF)$
such that $\{v\} \in \bF$ for all $v \in F^{0}$ and
$C^{*}(E,\CL_E,\bE) \cong C^{*}(F,\CL_F,\bF)$.

\vskip 1pc

\begin{ex}\label{ex_mergedgraph1}
Consider the following  labelled graphs:
\vskip 1pc
\hskip .5pc
\xy
/r0.38pc/:
 (0,0)*+{\bullet}="V3";
 (14,0)*+{\bullet}="V4";
 "V3";"V3"**\crv{(0,0)&(-4,4)&(-8,0)&(-4,-4)&(0,0)};?>*\dir{>}\POS?(.5)*+!D{};
 "V4";"V4"**\crv{(14,0)&(18,4)&(22,0)&(18,-4)&(14,0)};?>*\dir{>}\POS?(.5)*+!D{};
 "V3";"V4"**\crv{(0,0)&(8,4)&(14,0)};?>*\dir{>}\POS?(.5)*+!D{};
 "V4";"V3"**\crv{(14,0)&(6,-4)&(0,0)};?>*\dir{>}\POS?(.5)*+!D{};
 (-4,4)*+{0},(18,4)*+{1},(7,4)*+{0},(7,-4)*+{1},
 (0,-2)*+{v_1},(14,-2)*+{v_2},
 (-16,0)*+{(E_1,\CL_1)},

\endxy
\vskip .5pc
\hskip .5pc
\xy
/r0.38pc/:
(0,0)*+{\bullet}="V1";
 (14,0)*+{\bullet}="V2";
 "V1";"V1"**\crv{(0,0)&(-4,4)&(-8,0)&(-4,-4)&(0,0)};?>*\dir{>}\POS?(.5)*+!D{};
 "V2";"V2"**\crv{(14,0)&(18,4)&(22,0)&(18,-4)&(14,0)};?>*\dir{>}\POS?(.5)*+!D{};
 "V1";"V2"**\crv{(0,0)&(8,4)&(14,0)};?>*\dir{>}\POS?(.5)*+!D{};
 "V2";"V1"**\crv{(14,0)&(6,-4)&(0,0)};?>*\dir{>}\POS?(.5)*+!D{};
 (-4,4)*+{0},(18,4)*+{0},(7,4)*+{1},(7,-4)*+{1},
 (0,-2)*+{v_1},(14,-2)*+{v_2},
 (-16,0)*+{(E_2,\CL_2)}
\endxy
\vskip .5pc
\hskip .5pc
\xy
/r0.38pc/:
 (7,0)*+{\bullet}="V0";
 "V0";"V0"**\crv{(7,0)&(11,4)&(15,0)&(11,-4)&(7,0)};?>*\dir{>}\POS?(.5)*+!D{};
 "V0";"V0"**\crv{(7,0)&(3,4)&(-1,0)&(3,-4)&(7,0)};?>*\dir{>}\POS?(.5)*+!D{};
 (3,4)*+{0},(11,4)*+{1},
 (7,-2)*+{v},
 (-9,0)*+{(F,\CL_F)},
\endxy

\vskip 1pc
\noindent  Note that $\bE_i =\{\emptyset,\,\{v_1,\, v_2\}\}$ and
 $\{v_j\}\notin \bE_j$  for $i,j=1,2$
while  $\{v\}\in \bF$ for $v\in F^0$.
The $C^*$-algebras $C^*(E_i,\CL_i,\bE_i)$, $i=1,2$, and
$C^*(F,\CL_F,\bF)$  are all isomorphic to
 the Cuntz algebra $\mathcal{O}_2$.
\end{ex}

\vskip 1pc

\begin{ex}\label{ex_mergedgraph2}
The $C^*$-algebras $C^*(E, \CL_E, \bE)$ and $C^*(F, \CL_F, \bF)$
associated to the following labelled graphs
are isomorphic.
\vskip 1pc
\hskip .5pc
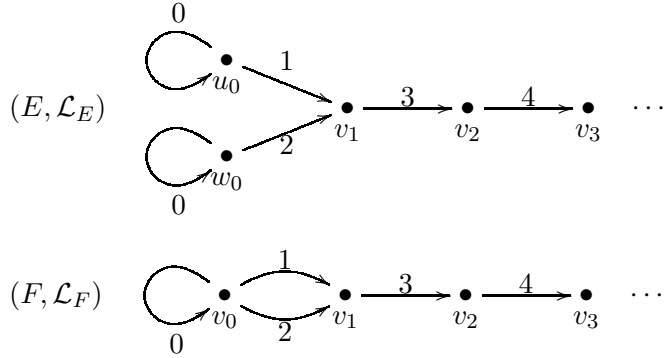
\xy
/r0.38pc/:
 (0,0)*+{\bullet}="V1";
 (0,8)*+{\bullet}="V2";
 (10,4)*+{\bullet}="V3";
 (20,4)*+{\bullet}="V4";
 (30,4)*+{\bullet}="V5";
 "V1";"V1"**\crv{(0,0)&(-4,4)&(-8,0)&(-4,-4)&(0,0)};?>*\dir{>}\POS?(.5)*+!D{};
 "V2";"V2"**\crv{(0,8)&(-4,12)&(-8,8)&(-4,4)&(0,8)};?>*\dir{>}\POS?(.5)*+!D{};
 "V1";"V3"**\crv{(0,0)&(10,4)};?>*\dir{>}\POS?(.5)*+!D{};
 "V2";"V3"**\crv{(0,8)&(10,4)};?>*\dir{>}\POS?(.5)*+!D{};
 "V3";"V4"**\crv{(10,4)&(20,4)};?>*\dir{>}\POS?(.5)*+!D{};
 "V4";"V5"**\crv{(20,4)&(30,4)};?>*\dir{>}\POS?(.5)*+!D{};
 (-14,4)*+{(E,\CL_E)},
 (0,-2)*+{w_0}, (0,6)*+{u_0},(10,2)*+{v_1}, (20,2)*+{v_2}, (30,2)*+{v_3},
 (-4,-4)*+{0},(-4,12)*+{0},(5,8)*+{1},(5,1)*+{2},(15,5)*+{3},(25,5)*+{4},(35,4)*+{\cdots},
 \endxy
\vskip .5pc
\hskip .5pc
 \xy
 /r0.38pc/:
 (0,4)*+{\bullet}="V6";
 (10,4)*+{\bullet}="V7";
 (20,4)*+{\bullet}="V8";
 (30,4)*+{\bullet}="V9";
 "V6";"V6"**\crv{(0,4)&(-4,8)&(-8,4)&(-4,0)&(0,4)};?>*\dir{>}\POS?(.5)*+!D{};
 "V6";"V7"**\crv{(0,4)&(5,0)&(10,4)};?>*\dir{>}\POS?(.5)*+!D{};
 "V6";"V7"**\crv{(0,4)&(5,8)&(10,4)};?>*\dir{>}\POS?(.5)*+!D{};
 "V7";"V8"**\crv{(10,4)&(20,4)};?>*\dir{>}\POS?(.5)*+!D{};
 "V8";"V9"**\crv{(20,4)&(30,4)};?>*\dir{>}\POS?(.5)*+!D{};
 (-4,0)*+{0},(5,7)*+{1},(5,1)*+{2},(15,5)*+{3},(25,5)*+{4},(35,4)*+{\cdots},
 (-14,4)*+{(F,\CL_F)},
 (0,2)*+{v_0}, (10,2)*+{v_1}, (20,2)*+{v_2}, (30,2)*+{v_3}
 \endxy
\vskip 1pc
\noindent
 Note here that $\{v_i\}\in \bF$ for all $i=0, 1,2,\dots$,
 but $\{u_0\},\,\{w_0\}\notin \bE$.
\end{ex}

\vskip 1pc

\begin{dfn}\label{def-mergedgraph}
Let $(E,\CL_E,\bE)$ be a labelled space.
Then $v \sim w$ if and only if $[v]_l = [w]_l$ for all $l \geq 1$
defines an equivalence relation $\sim$ on $E^{0}$.
Let $[v]_\infty$ denote the equivalence class $\{ w\in E^0 \,:\, w \sim v \}$ of $v$
and let
$$F^{0} := E^{0}/\sim \, = \{[v]_\infty \,:\, v \in E^{0}\}.$$
If $\ld\in E^{1}$ is an edge such that $s(\ld) \in [v]_\infty$, $r(\ld) \in [w]_\infty$,
then draw an edge $e_{\ld}$ from $[v]_\infty$ to $[w]_\infty$
and label $e_{\ld}$ with $\CL_F(e_{\ld}): = \CL_E(\ld)$.
If $\ld_1,\, \ld_2  \in E^{1}$ are edges with
$s(\ld_i) \in [v]_\infty$, $r(\ld_i) \in [w]_\infty$, $i=1,2$,
and $\CL_E(\ld_1)=\CL_E(\ld_2)$,
we identify  $e_{\ld_1}$ with $e_{\ld_2}$.
Then $F=(F^0, F^{1} := \{ e_{\ld} \,:\, \ld \in E^{1} \})$ is a graph
with the range, source maps given by $r(e_{\ld}):=[r(\ld)]_\infty$,
$s(e_{\ld}):=[s(\ld)]_\infty$, respectively,
and  $(F, \CL_F)$ is called the {\it merged labelled graph} of $(E,\CL_E)$
(cf. \cite{LM}).
\end{dfn}

\vskip 1pc

  In Example~\ref{ex_mergedgraph1},
$(F,\CL_F)$ is the merged labelled graph of  $(E_i,\CL_i)$
with $v=[v_1]_\infty=[v_2]_\infty$,
and similarly in Example~\ref{ex_mergedgraph2},
$(F,\CL_F)$ is a merged graph of $(E,\CL_E)$
with $v_0=[u_0]_\infty=[w_0]_\infty$.
  Note that a path $e_{\ld} \in F^{\geq 1}$ does not arise
  from a path $\ld \in E^{\geq 1}$ in general as
  we can observe in Example~\ref{ex_mergedgraph3} below;
$e_{\ld_1\ld_4}=e_{\ld_1}e_{\ld_4} \in F^{2}$, but $\ld_1 \ld_4 \notin E^{2}$.

\vskip 1pc

\begin{remark}\label{remark-vertex}
Let $(E,\CL_E,\CB)$ be a labelled space such that
for each $v\in E^0$ there exists an $l\geq 1$
for which $[v]_l$ is finite. We have the following:
\begin{enumerate}
\item[(i)] $[v]_\infty \in \bE$ for all $v\in E^0$;
 if $[v]_l$ is a finite set, there exists an $l'\geq l$ such that
$[v]_l \supset [v]_{l+1} \supset \cdots \supset \ [v]_{l'}=[v]_{l'+1} =\cdots,$
hence $[v]_\infty=[v]_{l'} \in \bE$.
\item[(ii)] For $u,v\in E^0$, either $[u]_\infty=[v]_\infty$ or
$[u]_\infty \cap [v]_\infty=\emptyset$.
\end{enumerate}
\end{remark}

\vskip 1pc

\begin{ex}\label{ex_mergedgraph3}
Consider the  labelled graph $(E, \CL:=\CL_E)$.

\vskip 1pc
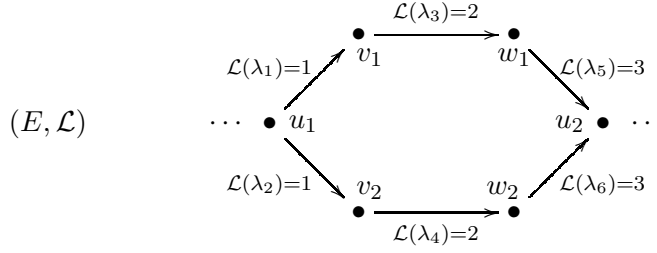
\xy
/r0.35pc/:
 (20,0)*+{\bullet}="V1";
 (28,8)*+{\bullet}="V2";
 (28,-8)*+{\bullet}="V3";
 (42,8)*+{\bullet}="V4";
 (42,-8)*+{\bullet}="V5";
 (50,0)*+{\bullet}="V6";
 "V1";"V2"**\crv{(20,0)&(28,8)};?>*\dir{>}\POS?(.5)*+!D{};
 "V1";"V3"**\crv{(20,0)&(28,-8)};?>*\dir{>}\POS?(.5)*+!D{};
 "V2";"V4"**\crv{(28,8)&(42,8)};?>*\dir{>}\POS?(.5)*+!D{};
 "V3";"V5"**\crv{(28,-8)&(42,-8)};?>*\dir{>}\POS?(.5)*+!D{};
 "V4";"V6"**\crv{(42,8)&(50,0)};?>*\dir{>}\POS?(.5)*+!D{};
 "V5";"V6"**\crv{(42,-8)&(50,0)};?>*\dir{>}\POS?(.5)*+!D{};
 (23,0)*+{u_1},(29,6)*+{v_1},(29,-6)*+{v_2},(42,6)*+{w_1},(41,-6)*+{w_2},(47,0)*+{u_2},
 (20,5)*+{{}_{\CL(\ld_1)=1}},(20,-5.5)*+{{}_{\CL(\ld_2)=1}},(35,10)*+{{}_{\CL(\ld_3)=2}},
 (35,-10)*+{{}_{\CL(\ld_4)=2}},(50,5)*+{{}_{\CL(\ld_5)=3}},(50,-5.5)*+{{}_{\CL(\ld_6)=3}},
 (0,0)*+{(E,\CL)},(16,0)*+{\cdots},(54,0)*+{\cdots},
 \endxy
\vskip .5pc

\noindent The merged labelled graph $(F,\CL_F)$ of $(E,\CL)$ is as follows.
\vskip .5pc

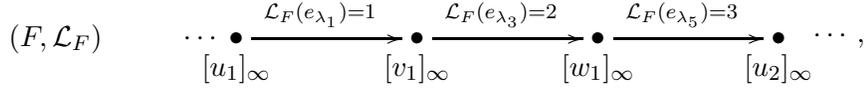
 \xy
 /r0.38pc/:
 (20,0)*+{\bullet}="V7";
 (35,0)*+{\bullet}="V8";
 (50,0)*+{\bullet}="V9";
 (65,0)*+{\bullet}="V10";
 "V7";"V8"**\crv{(20,0)&(35,0)};?>*\dir{>}\POS?(.5)*+!D{};
 "V8";"V9"**\crv{(35,0)&(50,0)};?>*\dir{>}\POS?(.5)*+!D{};
 "V9";"V10"**\crv{(50,0)&(65,0)};?>*\dir{>}\POS?(.5)*+!D{};
 (27,2)*+{{}_{\CL_F(e_{\ld_1})=1}},(42,2)*+{{}_{\CL_F(e_{\ld_3})=2}},
 (57,2)*+{{}_{\CL_F(e_{\ld_5})=3}},
 (5,0)*+{(F,\CL_F)},
 (20,-2.5)*+{[u_1]_\infty},(35,-2.5)*+{[v_1]_\infty},(50,-2.5)*+{[w_1]_\infty},
 (65,-2.5)*+{[u_2]_\infty},(17,0)*+{\cdots},(70,0)*+{\cdots,},
 \endxy
\vskip 1pc

\noindent where
$e_{\ld_1}=e_{\ld_2},\, e_{\ld_3}=e_{\ld_4},\, e_{\ld_5}=e_{\ld_6}$,
 $[v_i]_\infty =\{v_1,v_2\}$, and $ [w_i]_\infty =\{w_1,w_2\}$, $i=1,2$.
\end{ex}

\vskip 1pc

\begin{dfn}
Let $(F,\CL_F)$ be the merged labelled graph of $(E,\CL_E)$.
For $A \subset E^{0}$, $B \subset F^{0}$, we define
$[A]_\infty \subset F^{0}$, $ \widehat{B}\subset E^{0}$ by
$$ [A]_\infty :=\{ [v]_\infty \,:\, v \in A \}, \ \
 \widehat{B} : =\{v \,:\, [v]_\infty \in B \}.$$
\end{dfn}

\vskip 1pc

\noindent
Note that $[A_1\cap A_2]_\infty \subset [A_1]_\infty \cap [A_2]_\infty $
 and  $[A_1\cup A_2]_\infty = [A_1]_\infty \cup [A_2]_\infty $
whenever $A_1, A_2 \subset E^{0}$,
and for $A \subset E^{0}$ and $B \subset F^{0}$,
\begin{eqnarray}\label{eqn_E_0F_0}
A \subset \widehat{[A]_\infty}\ \text{ and }\  B = [\widehat{B}]_\infty.
\end{eqnarray}

\vskip 1pc

\begin{lem}\label{lem_samepaths}
 Let $(E, \CL_E, \bE)$ be a labelled space such that
for each $v\in E^0$, $[v]_l$ is finite for some $l\geq 1$,
and let $(F, \CL_F)$ be the merged labelled graph of $(E, \CL_E)$.
Then
\begin{eqnarray}\label{eqn_samepaths}
\CL_E([u]_\infty E^kv)=\CL_F([u]_\infty F^k[v]_\infty)
\end{eqnarray}
for all $k\geq 1$ and $u,v\in E^0$.
Moreover we have the following:
\begin{enumerate}
\item[(i)] $r(\af)=\widehat{r_F(\af)}$ and $[r(\af)]_\infty=r_F(\af)$
 for $\af\in  \CL^*(E)$.

\vskip .5pc
\item[(ii)] $s(\af) \subset \widehat{s_F(\af)}$ and $[s(\af)]_\infty=s_F(\af)$
 for $\af\in  \CL^*(E)$.

\vskip .5pc
\item[(iii)] $[\,[v]_l\,]_\infty=[\,[v]_\infty\,]_l$ for $v \in E^{0}$, $l\geq 1$.

\vskip .5pc
\item[(iv)] $[A\cap B]_\infty = [A]_\infty \cap [B]_\infty$
for $A, B \in \bE$.

\vskip .5pc
\item[(v)] $A=\widehat{[A]_\infty}$ for $A\in \bE$.
\end{enumerate}
\end{lem}

\begin{proof}
For simplicity of notation, we write  $\CL$
 for $\CL_E$ omitting the subscript $E$.
Note that each $[u]_\infty\in F^0$ is also  a subset of $E^0$ so that
an expression like $[u]_\infty E^k$ has obvious meaning.
 Since $\CL([u]_\infty E^kv)\subset \CL_F([u]_\infty F^k[v]_\infty )$ is clear,
we only need to show the reverse inclusion for (\ref{eqn_samepaths})
when $k\geq 1$.

Let $k=1$.
If $e_{\ld}\in  [u]_\infty F^1[v]_\infty$
and $\CL(e_{\ld}) = \af$,
 $\ld \in E^1$ is an edge such that  $s(\ld) \in [u]_\infty$,
 $r(\ld) \in [v]_\infty $ and $\CL(\ld) = \af$.
 Since $[v]_\infty = [r(\ld)]_\infty$, there exists an edge $\ld' \in E^1$ with
 $r(\ld') = v$ and $\CL(\ld')=\af$.
We claim that $[s(\ld')]_\infty=[s(\ld)]_\infty$.
 Since $[s(\ld)]_\infty \in\bE$ by Remark~\ref{remark-vertex}.(i),
 $r(\ld)\in r([s(\ld)]_\infty, \af)\in\bE$
 hence $[r(\ld)]_\infty\subset r([s(\ld)]_\infty, \af)$.
 Similarly, $v=r(\ld')\in r([s(\ld')]_\infty,\af)\in \bE$
 implies that $[v]_\infty\subset r([s(\ld')]_\infty,\af)$.
 Suppose $[s(\ld)]_\infty\neq [s(\ld')]_\infty$.
 Then $[s(\ld)]_\infty \cap [s(\ld')]_\infty=\emptyset$ by Remark~\ref{remark-vertex}.(ii),
 since $(E,\CL,\bE)$ is weakly left-resolving,
 $[v]_\infty=[r(\ld)]_\infty \subset r\big([s(\ld)]_\infty,\af)\cap r([s(\ld')]_\infty,\af\big)
 =r\big([s(\ld)]_\infty\cap [s(\ld')]_\infty,\af\big)=\emptyset,$
 a contradiction.
 Thus  $[s(\ld)]_\infty= [s(\ld')]_\infty$,
 namely $s(\ld')\in [u]_\infty$, and we have
\begin{eqnarray}\label{eqn_k=1}
\CL([u]_\infty E^1v)= \CL_F([u]_\infty F^1[v]_\infty).
\end{eqnarray}
 Now let $k=2$ and $e_{\ld_1} e_{\ld_2}\in F^2$ be a path with
$[u]_\infty := s_F(e_{\ld_1})$,
$[w]_\infty :=r_F(e_{\ld_1})=s_F(e_{\ld_2})$,
$[v]_\infty := r_F(e_{\ld_2})$, and
$\CL_F(e_{\ld_1} e_{\ld_2})=\af_1 \af_2$.
Then by (\ref{eqn_k=1}), there exist  $\ld_1',\, \ld_2'\in E^1$
such that
\begin{align*} & s(\ld_2')\in [w]_\infty, \ r(\ld_2')=v,\   \CL(\ld_2')=\af_2, \\
&  s(\ld_1')\in [u]_\infty, \ r(\ld_1')=s(\ld_2'),\   \CL(\ld_1')=\af_1.
\end{align*}
Then $\ld=\ld_1' \ld_2'\in [u]_\infty E^2v$ is a path with
$e_{\ld_1'}e_{\ld_2'}=e_{\ld_1}e_{\ld_2}$
and $\CL(\ld)=\CL_F(e_{\ld_1}e_{\ld_2})=\af_1 \af_2$.
Thus $\CL([u]_\infty E^2v)= \CL_F([u]_\infty F^2[v]_\infty)$.
For $k\geq 3$, one can repeat the process inductively.
Moreover (\ref{eqn_samepaths}) implies that
\begin{eqnarray}\label{eqn-CECF}
\CL(E^kv)= \CL_F(F^k[v]_\infty), \ k\geq 1.
\end{eqnarray}

  (i) To show   $r(\af)=\widehat{r_F(\af)}$ for $\af \in \CL^{*}(E)$,
let $v \in r(\af)$.
Then there exists $\ld \in E^{\geq 1}$
such that $r(\ld)=v$ and $\CL(\ld)=\af$.
The edge $e_\ld \in F^{\geq 1}$ has the range vertex $r(e_\ld)=[v]_\infty$
and the label $\CL_F(e_\ld)=\af$.
Hence $[v]_\infty  \in r_F(\af)$, namely $v \in \widehat{r_F(\af)}$.
Conversely, if $v \in \widehat{r_F(\af)}$,
that is, $[v]_\infty \in r_F(\af)$,
by (\ref{eqn-CECF}), there is a path $\ld \in E^{\geq 1}$ with $\CL(\ld)=\af$, $r(\ld)=v$.
Hence $v \in r(\af)$.
Also, by (\ref{eqn_E_0F_0}) we have
$[r(\af)]_\infty =[\,\widehat{r_F(\af)}\,]_\infty =r_F(\af)$.

  (ii) Since $s(\af) \subset \widehat{s_F(\af)}$ is clear,
we have $[s(\af)]_\infty  \subset [\,\widehat{s_F(\af)}\,]_\infty =s_F(\af)$
by (\ref{eqn_E_0F_0}).
Also (\ref{eqn_samepaths}) shows that  $[s(\af)]_\infty  \supset s_F(\af)$.

  (iii) The equality
$[\, [v]_l\,]_\infty =[\,[v]_\infty\, ]_l$
follows  from
\begin{align*}
w\in [v]_l & \Longleftrightarrow [w]_l = [v]_l\\
& \Longleftrightarrow \CL(E^{\leq l} w)=\CL(E^{\leq l} w)\\
& \Longleftrightarrow \CL_F(F^{\leq l} [w]_\infty)=\CL_F(F^{\leq l} [v]_\infty)
\ \,(\text{by (\ref{eqn-CECF})})\\
& \Longleftrightarrow
[w]_\infty \in [\,[v]_\infty\,]_l.
\end{align*}

  (iv) It suffices to show that
  $[A]_\infty  \cap [B]_\infty  \subset [A \cap B]_\infty $.
Let $[v_1]_\infty =[v_2]_\infty  \in [A]_\infty  \cap [B]_\infty $
for some $v_1 \in A$, $v_{2} \in B$.
Since $A, B \in \bE$, there exists $l\geq 1$
such that $[v_{1}]_\infty  \subset [v_{1}]_l \subset A$
and $[v_{2}]_\infty  \subset [v_{2}]_l \subset B$.
Hence $[v_{1}]_l=[v_{2}]_l$ and so
 $v_1 \in A \cap B$ and  $[v_1]_\infty  \in [A \cap B]_\infty $.

 (v) $A\subset \widehat{[A]_\infty }$ is clear.
 If $v\in \widehat{[A]_\infty }$,
 $[v]_\infty \in [A]_\infty $  and so $[v]_\infty=[w]_\infty$ for some $w\in A$.
 Writing $A=\cup_j [w_j]_l\in \bE$, we have $w\in [w_j]_l$ for some $j$,
 then $v\in [w_j]_l\subset A$ because $v\sim w\sim_l w_j$.
\end{proof}

\vskip 1pc

If for each $v\in E^0$, $[v]_l$ is finite for some $l\geq 1$,
then $[v]_k=[v]_\infty$ for some $k\geq l$, and
so every $A\in \bE$ can be written as $A=\cup_j [v_j]_l$
with $[v_j]_l=[v_j]_\infty$ and thus by Lemma~\ref{lem_samepaths}.(iii),
we have
$$[A]_\infty =[\,\cup_j [v_j]_l\,]_\infty
=\cup_j [\,[v_j]_l\,]_\infty=\cup_j [\,[v_j]_\infty\,]_l \in \bF. $$

\vskip 1pc

\begin{prop}\label{prop_bijection}
Let $(E,\CL,\bE)$ be a labelled space such that if $v\in E^0$,
$[v]_l$ is finite for some $l\geq 1$.
Then the map $A\mapsto [A]_\infty:  \bE \rightarrow \bF$ is a bijection such that
$[r(A,\af)]_\infty = r_F([A]_\infty,\af)$ for $A\in \bE$, $\af\in \CL^*(E)$.
\end{prop}

\begin{proof}
 To show that the map is surjective,
 let $B=\cup_j [\,[v_j]_\infty\,]_l\in \bF$.
 Then by Lemma~\ref{lem_samepaths}.(iii),
 $$\widehat{B}=\{v:[v]_\infty\in \cup_j \, [\,[v_j]_\infty\,]_l \}
 =\{v:[v]_\infty\in \cup_j\, [\,[v_j]_l\,]_\infty \}
 =\cup_j\, [v_j]_l\in \bE $$
 and  $B=[\widehat{B}]_\infty$ by (\ref{eqn_E_0F_0}).

 For injectivity,
 let  $[A_1]_\infty=[A_2]_\infty$, $A_1,\, A_2\in \bE$.
 Then by  Lemma~\ref{lem_samepaths}.(v),
$A_1=\widehat{[A_1]_\infty}=\widehat{[A_2]_\infty}=A_2$.

Now we show that
$[r(A,\af)]_\infty = r_F([A]_\infty,\af)$ for $A\in \bE$, $\af\in \CL^*(E)$.
Clearly $[r(A,\af)]_\infty \subset r_F([A]_\infty,\af)$ holds.
If $[v]_\infty\in r_F([A]_\infty,\af)$,
there exists a path $e_{\ld_1}\cdots e_{\ld_n}\in F^n$ with
$r_F(e_{\ld_1}\cdots e_{\ld_n})=[v]_\infty$.
Let $[u]_\infty:=s_F(e_{\ld_1}\cdots e_{\ld_n})$,  $[u]_\infty\in [A]_\infty$,
and $\af=\CL_F(e_{\ld_1}\cdots e_{\ld_n})$.
We may assume that $u\in A$ since $[u]_\infty=[u']_\infty$ for some $u'\in A$.
  By (\ref{eqn_samepaths}), we can find a path $\ld\in E^n$ with
  $r(\ld)=v$, $s(\ld)\in [u]_\infty\,(\subset A)$, and $\CL(\ld)=\af$.
  Thus $v\in r([u]_\infty,\af)\subset r(A, \af)\,(\in \bE)$.
  Then $[v]_\infty\subset  r(A, \af)$ and we conclude that
  $[v]_\infty\in [r(A, \af)]_\infty$.
\end{proof}

\vskip 1pc

Clearly the merged labelled space $(F, \CL, \bF)$ has no sinks
since we assume that  $(E, \CL, \bE)$ has no sinks.
Besides, $(F, \CL, \bF)$ has the following properties.

\vskip 1pc

\begin{prop}
Let $(E,\CL,\bE)$ be a labelled space such that if $v\in E^0$,
 $[v]_l$ is finite for some $l\geq 1$.
Then the merged labelled space $(F, \CL, \bF)$ is set-finite 
and receiver set-finite, respectively
if and only if  $(E,\CL,\bE)$ is set-finite and receiver set-finite,
respectively. 
Moreover $(F, \CL, \bF)$ is weakly left-resolving whenever 
$(E,\CL,\bE)$ is weakly left-resolving.
\end{prop}

\begin{proof}
By Lemma~\ref{lem_samepaths}, we know that
$A\mapsto [A]_\infty:\bE\to \bF$ forms a bijection.
From the following equalities (using (\ref{eqn-CECF}))
\begin{align*}
\CL(E^lA)&=\cup_{v\in A}\,\CL(E^l v)= \cup_{v\in A}\,\CL_F(F^l [v]_\infty)\\
 &= \cup_{[v]_\infty\in [A]_\infty}\,\CL_F(F^l [v]_\infty)\\
 & =\CL_F(F^l [A]_\infty\,)
\end{align*}
we have that $(E,\CL,\bE)$ is  receiver set-finite
if and only if  $(E,\CL,\bE)$ is  receiver set-finite.
 Since  $[r(A,\af)]_\infty = r_F([A]_\infty,\af)$ (Proposition~\ref{prop_bijection}),
\begin{align*}
\CL([A]_\infty F^{1})
 & = \{a \in \CA  \,:\, r_F([A]_\infty,a) \neq \emptyset \}\\
 & =\{a \in \CA \,:\, [r(A,a)]_\infty \neq \emptyset \}\\
 & =\CL(AE^{1}),
\end{align*}
which proves the equivalence of set-finiteness of $(E,\CL,\bE)$ and $(F, \CL, \bF)$.

 Since $(E,\CL,\bE)$ is weakly left-resolving, 
 by Lemma~\ref{lem_samepaths}.(iv) and Proposition~\ref{prop_bijection},
 we have
 \begin{align*}
  & r(A,\af) \cap r(B,\af) = r(A\cap B,\af)\\
 \Longleftrightarrow & \ [r(A,\af) \cap r(B,\af)]_\infty = [r(A\cap B,\af)]_\infty  \\
  \Longleftrightarrow & \ r_F([A]_\infty ,\af)\cap r_F([B]_\infty ,\af)
 =r_F([A]_\infty \cap [B]_\infty ,\af\,).
 \end{align*}
 Thus  $(F, \CL, \bF)$ is weakly left-resolving.
\end{proof}

\vskip 1pc

\begin{thm}\label{theorem-merged-iso}
Let $(E,\CL,\bE)$ be a labelled space such that if $v\in E^0$,
 $[v]_l$ is finite for some $l\geq 1$, and
let $(F, \CL, \bF)$  the merged labelled space of $(E,\CL,\bE)$.
Then $\{[v]_\infty \}\in \bF$ for every vertex $[v]_\infty \in F^0$ and
$$C^{*}(E, \CL, \bE) \cong C^{*}(F, \CL, \bF).$$
\end{thm}

\begin{proof}
 For $v\in E^0$, let $[v]_l$ be finite with $[v]_\infty =[v]_l$.
 We see from the proof of Lemma~\ref{lem_samepaths}.(iii)
 that $[w]_l=[v]_l(=[v]_\infty)
 \Longleftrightarrow [\,[w]_\infty\,]_l=[\,[v]_\infty\,]_l,$
 which then shows
 $\{[v]_\infty \} =\{[w]_\infty: [\,[w]_\infty\,]_l=[\,[v]_\infty\,]_l\}
 = [\,[v]_\infty\, ]_l \in \bF$.

Let $C^*(E, \CL, \bE)=C^*(p_{A}, s_{a})$
and $C^*(F, \CL_F, \bF)=C^*(q_{[A]_\infty }, t_{a})$.
Note first that $\{P_A:=q_{[A]_\infty}: \,A\in \bE\}\cup \{S_a:=t_a:\, a\in \CA\}$
is a representation of $(E, \CL, \bE)$:
\begin{enumerate}
\item[(i)] If $A,B\in \bE$,
then $P_A P_B=q_{[A]_\infty}\,q_{[B]_\infty}=q_{\,[A]_\infty\cap [B]_\infty}
=q_{\,[A\cap B]_\infty}=P_{A\cap B}$
 and $P_{A\cup B}= q_{\,[A\cup B]_\infty}=q_{\,[A]_\infty\cup [B]_\infty}
 =q_{[A]_\infty}+q_{[B]_\infty}-q_{\,[A\cap B]_\infty}
 =P_A + P_B - P_{A\cap B}$, where $P_\emptyset=q_{\emptyset}=0$.

\item[(ii)] If $A \in \bE$ and  $a \in \CA$, then
  $P_A S_a =q_{[A]_\infty} t_a = t_a q_{r_F([A]_\infty,a)}
  = t_a q_{[r(A,a)]_\infty}= S_a P_{r(A, a)}$.

\item[(iii)] If $a,b\in \CA$, $S_a^* S_a= t_{a}^{*}t_{a} =q_{r_F(a)}
=q_{[r(a)]_\infty}=P_{r(a)}$ and
 $S_a^* S_b= t_{a}^{*}t_{b}=0$ unless $a = b$.

\item[(iv)] For $A\in \bE$,
  \begin{align*} P_A   = q_{[A]_\infty} &
  =  \sum_{a \in \CL_F([A]_\infty F^1)}  t_{a}q_{r_F([A]_\infty,a)}t_{a}^{*} \\
       & = \displaystyle{\sum_{a \in \CL(AE^{1})}}\, t_a q_{\,[r(A,a)]_\infty}\,t_a^*\\
       & = \displaystyle{\sum_{a \in \CL(AE^{1})}} S_aP_{r(A,a)}\,S_a^*.
  \end{align*}
\end{enumerate}
Thus there exists a surjective $*$-homomorphism $\Phi : C^*(E, \CL, \bE)  \to C^*(F, \CL, \bF)$
such that $\Phi(p_A)=q_{[A]_\infty}$  and $\Phi(s_a)=t_a$
for $A\in \bE,\  a\in \CA$.
$\Phi$ is an isomorphism by Theorem~\ref{theorem-gaugeinvariant}.
\end{proof}

\vskip 1pc

Recall \cite{BP2, JK} that a labelled space $(E,\CL,\bE)$ is {\it disagreeable} if
for each $[v]_l$,
there exists an $N\geq 1$ such that for all $n\geq N$
there is a labelled path
$\af\in \CL([v]_lE^{\geq n})$ that is not agreeable,
that is, not of the form
$\af=\bt\af'=\af'\gm$ for some $\af',\, \bt,\, \gm \in \CL(E^{\geq 1})$
with $|\bt|=|\gm|\leq l$.
Also $(E, \CL, \CB)$ is  {\it strongly cofinal} \cite{JK} if
for all  $x\in  \CL(E^\infty)$, $w\in s(x)$, and $[v]_l\in \Omega_l(E)$,
there are $N\geq 1$ and  a finite number of labelled paths
$\ld_1, \dots, \ld_m$  such that
$$r([w]_1,  x_1\cdots x_N )\subset \cup_{i=1}^m r([v]_l, \ld_i).$$

\vskip 1pc

\begin{thm}\label{theorem-cofinal}
Let $(E,\CL,\bE)$ be a labelled space such that if $v\in E^0$,
 $[v]_l$ is finite for some $l\geq 1$, and let
 $(F, \CL_F, \bF)$ be the merged labelled space of $(E,\CL,\bE)$.
 Then we have the following:
 \begin{enumerate}
 \item[(i)] $(E, \CL, \bE)$ is strongly cofinal
 if and only if  $(F, \CL_F, \bF)$ is strongly cofinal.
 \item[(ii)]  $(E, \CL, \bE)$ is disagreeable
 if and only if $(F, \CL_F, \bF)$ is disagreeable.
 \end{enumerate}
\end{thm}

\begin{proof}
(i)  Suppose $(E, \CL, \bE)$ is strongly cofinal
and let $x=x_1x_2\cdots\in \CL_F(F^\infty)$,  $[u_0]_\infty\in s_F(x)$
and $[\,[v]_\infty\,]_l\in\Omega_l(F)$.
Fix $[u_i]_\infty \in r_F(x_i)$ for each $i$.
Then $x_1\cdots x_i \in \CL([u_0]_\infty F^i [u_i]_\infty)$ for $i\geq 1$.
Since $\CL([u_0]_\infty E^i u_i)=\CL_F([u_0]_\infty F^i [u_i]_\infty)$ (by (\ref{eqn_samepaths})),
 $x_1\cdots x_i \in \CL([u_0]_\infty E^i u_i)$ for all $i\geq 1$.
Then the finite set  $[u_0]_\infty $ must have a vertex
$u_0'\in [u_0]_\infty$ such that $x_1\cdots x_i \in \CL( u_0'E^i u_i)$
for infinitely many $i$'s, which means that
$x\in \CL(u_0' E^\infty)$.
Since $(E, \CL, \bE)$ is strongly cofinal,
there exists an $N\geq 1$ and
a finite number of labelled paths $\ld_1, \dots, \ld_m\in \CL(E^{\geq 1})$
such that $r([u_0']_1, x_1\cdots x_N)\subset \cup_{j=1}^m r([v]_l, \ld_j)$.
Then $[\,r([u_0']_1, x_1\cdots x_N)\,]_\infty
\subset [\,\cup_{j=1}^m r([v]_l, \ld_j)\,]_\infty $,
that is,
$$ r_F(\,[\,[u_0']_1\,]_\infty , x_1\cdots x_N\,)
\subset \cup_{j=1}^m r_F(\,[\,[v]_l\,]_\infty, \ld_j\,)
= \cup_{j=1}^m r_F(\, [\,[v]_\infty \,]_l, \ld_j\,),$$
and we see that $(F, \CL_F, \bF)$ is strongly cofinal.

Conversely, assuming that $(F, \CL_F, \bF)$ is strongly cofinal, if
$x\in \CL(E^\infty)$ is an infinite labelled path with
$u\in s(x)$ and $[v]_l\in \Omega_l(E)$,
clearly $x\in \CL_F([u]_\infty F^\infty)$ and so
there exist an $N\geq 1$ and a finite number of labelled paths
$\ld_1, \dots, \ld_m\in \CL_F(F^{\geq 1})$ such that
$$r_F(\,[\,[u]_\infty\, ]_1, x_1\cdots x_N\,)
\subset \cup_{j=1}^m r_F(\,[\,[v]_\infty \,]_l, \ld_j\,).$$
Hence, we have
$[r([u]_1, x_1\cdots x_N)]_\infty\,\subset [\cup_{j=1}^m r([v]_l, \ld_j)]_\infty$
by Proposition~\ref{prop_bijection}.
Then  Lemma~\ref{lem_samepaths}.(v) shows that
$r([u]_1, x_1\cdots x_N)\subset  \cup_{j=1}^m r([v]_l, \ld_j)$ and
so $(E, \CL, \bE)$ is strongly cofinal.

 (ii) Note that each  $[v]_l$
 is a union of finitely many  equivalence classes $[v']_\infty$ of $v'\in [v]_\infty$.
 From (\ref{eqn_samepaths}), we have
 \begin{align*}
 \CL([v]_l E^n)& =\cup_{v'\in [v]_l} \, \CL([v']_\infty E^n)\\
 & =\cup_{v'\in [v]_l,\, w\in E^0}\, \CL([v']_\infty E^n w)\\
 &=\cup_{[v']_\infty \in [\,[v]_l]_\infty,\,
    [w]_\infty \in F^0}\, \CL_F(\,[v']_\infty F^n [w]_\infty \,)\\
 &=\cup_{[v']_\infty \in [\,[v]_\infty \,]_l}\, \CL_F(\,[v']_\infty F^n\,)\\
 &= \CL_F(\,[\,[v]_\infty \,]_l F^n\,),
 \end{align*}
which shows the assertion.
\end{proof}

\vskip 1pc
It is known that if  $C^*(E, \CL, \bE)$ is simple,
$(E, \CL, \bE)$  is strongly cofinal (\cite[Theorem 3.8]{JK})
and  if, in addition, $\{v\}\in \bE$ for all $v\in E^0$,
$(E, \CL, \bE)$  is disagreeable (\cite[Theorem 3.14]{JK}).
Also if $(E, \CL, \bE)$  is strongly cofinal and disagreeable,
$C^*(E, \CL, \bE)$ is simple (\cite[Theorem 3.16]{JK}).
Therefore by Theorem~\ref{theorem-merged-iso} and Theorem~\ref{theorem-cofinal}
we have the following corollary.

\vskip 1pc

\begin{cor}\label{corollary}
Let $(E, \CL, \bE)$ be a set-finite, receiver set-finite, and  weakly left-resolving labelled space
such that for each $v\in E^0$ there is an $l\geq 1$ for which $[v]_l$ is finite.
Then  $C^{*}(E, \CL, \bE)$ is simple
if and only if $(E, \CL, \bE)$ is strongly cofinal and disagreeable.
\end{cor}


\vskip 1pc

\begin{thebibliography}{4}

\bibitem{BHRS} T. Bates, J. H. Hong, I. Raeburn and W.
Szymanski, {\em The ideal structure of the $C^*$-algebras of
infinite graphs}, Illinois J. Math., \textbf{46}(2002),
1159--1176.

\bibitem{BPRS} T. Bates, D. Pask, I. Raeburn, and W.
Szymanski, {\em The $C^*$-algebras of row-finite graphs}, New York
J. Math. \textbf{6}(2000), 307--324.


\bibitem{BP1} T. Bates and D. Pask, {\em $C^*$-algebras of labelled graphs},
J. Operator Theory. \textbf{57}(2007), 101--120.

\bibitem{BP2} T. Bates and D. Pask, {\em $C^*$-algebras of labelled graphs II - simplicity results},
  Math. Scand. \textbf{104}(2009), no. 2, 249--274.

\bibitem{BKR} S. Boyd, N. Keswani and I. Raeburn,
{\em Faitheful representations of crossed producs by endomorphisms},
  Proc. Amer. Math. Soc. \textbf{118}(1993), no. 2, 427--436.

\bibitem{CM} T. Carlsen and  K. Matsumoto, {\em Some remarks on the $C^*$-algebras
associated with subshifts},  Math. Scand. \textbf{95}(2004),
145--160.

\bibitem{EL} R. Exel and M. Laca, {\em Cuntz-Krieger algebras for infinite
matrices},
J. Reine. Angew. Math. \textbf{512}(1999), 119--172.

\bibitem{JK} J. A Jeong and S. H. Kim,
{\em On simple labelled graph $C^*$-algebras}, 2010, preprint (arXiv:1101.4739v1[math.OA]).

\bibitem{Ka} T. Katsura, {\em A class of $C^*$-algebras
generalizing both graph algebras and homeomorphism $C^*$-algebras I,
Fundamental results},
Trans. Amer. Math. Soc. \textbf{356}(2004), 4287--4322.

\bibitem{KMST} T. Katsura, P. Muhly, A. Sims  and M. Tomforde
{\em Ultragraph $C^*$-algebras via topological quivers},
Studia Math. \textbf{187}(2008), 137--155.

\bibitem{KPR} A. Kumjian, D. Pask and I. Raeburn, {\em Cuntz-Krieger
algebras of directed graphs}, Pacific J. Math. \textbf{184}(1998),
161--174.

\bibitem{KPRR} A. Kumjian, D. Pask, I. Raeburn  and J. Renault
{\em Graphs, groupoids, and Cuntz-Krieger algebras}, J. Funct.
Anal. \textbf{144}(1997), 505--541.


\bibitem{LM}  D. Lind and B. Marcus, {\em An introduction to
symbolic dynamics and coding}, Cambrige Univ. press, 1999.

\bibitem{Ma}  K. Matsumoto, {\em On $C^*$-algebras
associated with subshifts}, Intern. J. Math. \textbf{8}(1997),
457--374.

\bibitem{MT}  P. Muhly and M. Tomforde
{\em Topological quivers},
Intern. J. Math. \textbf{16}(2005), 693--755.

\bibitem{To1} M.Tomforde, {A unified approach to Exel-Laca algebras and
$C^*$-algebras associated to gtaphs}, J. Operator Theory
\textbf{50}(2003), 345--368.

\bibitem{To2} M.Tomforde, {Simplicity of ultragraph algebras}, Indiana Univ. Math. J.
\textbf{52}(2003), 901--926.

\end{thebibliography}
\end{document}